\pdfoutput=1

\documentclass[12pt,reqno]{amsart}
\usepackage{amsmath}
\usepackage{amsthm}
\usepackage{amssymb}
\usepackage{amsfonts}
\usepackage{mathrsfs}
\usepackage[backrefs]{amsrefs}
\usepackage{float}
\restylefloat{table}
\usepackage[toc]{appendix}

\usepackage{setspace}	

\title{Frobenius Groups with Perfect Order Classes}
\author{James McCarron}
\address{Maplesoft\\
615 Kumpf Drive,\\
Waterloo, ON, Canada\\
N2V 1K8}
\email{\texttt{james@maplesoft.com}}
\date{\today}

\swapnumbers
\theoremstyle{plain}
\newtheorem{theorem}{Theorem}[section]
\newtheorem*{theoremA}{Theorem A}
\newtheorem*{theoremB}{Theorem B}
\newtheorem*{theoremC}{Theorem C}
\newtheorem*{theoremD}{Theorem D}
\newtheorem*{theoremE}{Theorem E}

\newtheorem{proposition}[theorem]{Proposition}
\newtheorem{corollary}[theorem]{Corollary}
\newtheorem{lemma}[theorem]{Lemma}

\theoremstyle{remark}

\newtheorem{problem}[theorem]{Problem}
\newtheorem{remark}[theorem]{Remark}
\newtheorem{example}[theorem]{Example}

\numberwithin{theorem}{section}

\newcommand{\defn}[1]{\textbf{#1}}

\newcommand{\Maple}{\textsc{Maple}}
\newcommand{\GAP}{\textsc{GAP}}
\newcommand{\Magma}{\textsc{Magma}}

\let\eulerphi\phi

\newcommand{\integers}{\ensuremath{\mathbb{Z}}}

\newcommand{\intmod}[1]{\ensuremath{\integers/{#1}\integers}}

\newcommand{\abs}[1]{\ensuremath{\left\lvert #1\right\rvert}}
\newcommand{\order}[1]{\ensuremath{\abs{#1}}} 

\newcommand{\iso}{\ensuremath{\simeq}}

\newcommand{\centre}[1]{\ensuremath{\mathrm{Z}(#1)}}

\makeatletter
\def\moverlay{\mathpalette\mov@rlay}
\def\mov@rlay#1#2{\leavevmode\vtop{%
   \baselineskip\z@skip \lineskiplimit-\maxdimen
   \ialign{\hfil$\m@th#1##$\hfil\cr#2\crcr}}}
\newcommand{\charfusion}[3][\mathord]{
    #1{\ifx#1\mathop\vphantom{#2}\fi
        \mathpalette\mov@rlay{#2\cr#3}
      }
    \ifx#1\mathop\expandafter\displaylimits\fi}
\makeatother

\newcommand{\bigcupdot}{\charfusion[\mathop]{\bigcup}{\cdot}}

\newcommand{\sym}[1]{\ensuremath{\operatorname{S}_{#1}}}
\newcommand{\alt}[1]{\ensuremath{\operatorname{A}_{#1}}}
\newcommand{\dih}[1]{\ensuremath{\operatorname{D}_{#1}}}
\newcommand{\quat}[1]{\ensuremath{\operatorname{Q}_{#1}}}
\newcommand{\cyclic}[1]{\ensuremath{\operatorname{C}_{#1}}}
\newcommand{\GL}[2]{\ensuremath{\operatorname{GL}({#1},{#2})}}
\newcommand{\SL}[2]{\ensuremath{\operatorname{SL}({#1},{#2})}}

\DeclareMathOperator{\aut}{Aut}

\DeclareMathOperator{\Hol}{Hol}	

\newcommand{\abel}[1]{\ensuremath{{#1}^\mathrm{ab}}}

\let\sdp\rtimes 
\newcommand{\sdpf}{\ensuremath{\rtimes_{f}}} 

\DefineSimpleKey{bib}{myurl}
\newcommand\myurl[1]{\url{#1}}
\BibSpec{website}{%
	+{}{\PrintAuthors} {author}
	+{,}{ \textit} {title}
	+{}{ \parenthesize} {date}
	+{,}{ \myurl} {myurl}
	+{,}{ } {note}
	+{.}{ } {transition}
}

\newcommand{\ordclass}[2]{\ensuremath{{#1}^{[#2]}}} 
\newcommand{\ordcount}[2]{\ensuremath{f_{#1}(#2)}} 

\begin{document}

\begin{abstract}
The purpose of this paper is to investigate the
finite Frobenius groups with ``perfect order classes'';
that is, those for which the number of elements of each
order is a divisor of the order of the group.
If a finite Frobenius group has perfect order classes
then so too does its Frobenius complement, the Frobenius
kernel is a homocyclic group of odd prime power order,
and the Frobenius complement acts regularly on the elements
of prime order in the Frobenius kernel.
The converse is also true.

Combined with elementary number-theoretic arguments,
we use this to provide characterisations of several
important classes of Frobenius groups.
The insoluble Frobenius groups with perfect order classes
are fully characterised.
These turn out to be the perfect Frobenius groups whose
Frobenius kernel is a homocyclic $11$-group of rank $2$.

We also determine precisely which nilpotent Frobenius
complements have perfect order classes,
from which it follows that a Frobenius group with nilpotent
complement has perfect order classes only if the Frobenius
complement is a cyclic $\{2,3\}$-group of even order.

Those Frobenius groups for which the Frobenius complement
is a biprimary group are also described fully,
and we show that no soluble Frobenius group whose Frobenius
complement is a $\{2,3,5\}$-group with order divisible by $30$
has perfect order classes.
\end{abstract}

\maketitle
\tableofcontents

\section{Introduction}\label{sec:intro}


A finite group (and all our groups are supposed to be finite)
is said to have \defn{perfect order classes} if the
number of elements of any given order either is zero or is a divisor
of the order of the group.
The non-cyclic group of order $4$ does not have perfect order classes
since the number of elements of order $2$ is equal to $3$,
which is not a divisor of its order $4$.
However, the cyclic group of order $4$ does have perfect order
classes, because it has $2$ elements of order $4$ and one element
each of orders $1$ and $2$.

Groups with perfect order classes (often, ``subsets'') were introduced
by Finch and Jones in \cite{FinchJones2002}, and have subsequently
been studied by a number of authors.
See, for example,
\cites{FinchJones2003,JonesToppin2011,Das2009b,Shen2010,ShenShiShi2013,TuanHai2010,FordKonyaginLuca2012}.


The object of this paper is to investigate
Frobenius groups with perfect order classes.
Our first main result reduces the study of Frobenius groups with
perfect order classes to that of Frobenius complements with
perfect order classes.

\begin{theoremA}
Let $G = K\sdpf H$ be a Frobenius group with Frobenius kernel $K$
and Frobenius complement $H$.
Then $G$ has perfect order classes if, and only if,
the following are true:
\begin{itemize}
\item[(a)]{$H$ has perfect order classes;}
\item[(b)]{$K$ is a homocyclic $p$-group for an odd prime number $p$; and,}
\item[(c)]{$\order{H} = p^r - 1$, where $r$ is the rank of $K$.}
\end{itemize}
\end{theoremA}

Condition (c) of this theorem is equivalent to the assertion that the Frobenius
complement $H$ acts regularly on the elements of order $p$ in $G$,
and it follows that they comprise a single conjugacy class in $G$.
We have chosen the arithmetic formulation for this condition because
it is most directly applicable in subsequent results.

This is advantageous because a great deal is known about the
structure of Frobenius complements.
For example, we are able to give a complete and precise
description of the insoluble Frobenius groups with perfect
order classes.

\begin{theoremB}
An insoluble Frobenius group has perfect order classes if,
and only if, it is perfect and its Frobenius kernel is
a homocyclic $11$-group of rank $2$.
\end{theoremB}
We remark that a Frobenius group is perfect if, and only if,
its Frobenius complement is isomorphic to the binary icosahedral group
of order $120$.

We now turn to the richer class of soluble Frobenius groups,
for which we are able to describe several important cases.
We are able to completely characterise Frobenius
groups with perfect order classes whose Frobenius complement
is nilpotent.
To state our result for this case, we note that a \defn{Pierpont prime}~\cites{Pierpont1895,A005109} is
a prime number $p$ of the form $p = 1 + 2^{\alpha}3^{\beta}$,
for some integers $\alpha, \beta\geq 0$.
(If $\beta = 0$, then $p$ is a Fermat prime,
provided that $\alpha > 0$.
But see \cites{BoklanConway2017}.)
At the time of writing it remains an open problem whether
there exist infinitely many Pierpont primes~\cite{WikiPierpontPrime}.

\begin{theoremC}
Let $G = K\sdpf H$ be a Frobenius group with nilpotent Frobenius complement $H$.
Then $G$ has perfect order classes if, and only if,
one of the following is true.
\begin{enumerate}
\item{$K$ is a cyclic $p$-group, for some Pierpont prime $p = 1 + \order{H} > 2$,
and $H$ is a cyclic $\{2,3\}$-group of even order; or}
\item{$G$ is isomorphic to one of the following groups:
\begin{enumerate}
\item{$(\cyclic{3^k}\times\cyclic{3^k})\sdpf\cyclic{8}$;}
\item{$(\cyclic{5^k}\times\cyclic{5^k})\sdpf\cyclic{24}$;}
\item{$(\cyclic{7^k}\times\cyclic{7^k})\sdpf\cyclic{48}$; or,}
\item{$(\cyclic{17^k}\times\cyclic{17^k})\sdpf\cyclic{288}$,}
\end{enumerate}
for some positive integer $k$.}
\end{enumerate}
\end{theoremC}

Observe that, in all cases, the Frobenius complement is cyclic.

It seems an interesting feature of our investigations that there
are Frobenius complements with perfect order classes that nevertheless
do not ``fit'' within any Frobenius group with perfect order classes.
We completely characterise the non-cyclic nilpotent Frobenius complements
with perfect order classes (there are two infinite series),
but show that none occurs in a Frobenius group with perfect order classes.
We also exhibit some soluble but non-nilpotent examples.

For the case in which the Frobenius complement is soluble but not
nilpotent we are able to characterise the Frobenius groups with
perfect order classes among those whose Frobenius complement is
biprimary (that is, whose order has just two prime divisors).


\begin{theoremD}
Let $G = K\sdpf H$ be a Frobenius group
and assume that the Frobenius complement $H$ is a non-nilpotent $\{2,q\}$-group,
where $q$ is an odd prime.
Then $G$ has perfect order classes if, and only if,
$q$ is equal either to $3$ or to $5$ and
$G$ has one of the following forms:
\begin{enumerate}
\item{$G\iso\cyclic{5^k}^2\sdpf\SL{2}{3}$;}
\item{$G\iso\cyclic{5^k}^2\sdpf(\cyclic{3}\sdp\cyclic{8})$;}
\item{$G\iso\cyclic{7^k}^2\sdpf(\cyclic{3}\sdp\cyclic{16})$;}
\item{$G\iso\cyclic{17^k}^2\sdpf(\cyclic{9}\sdp\cyclic{32})$; or,}
\item{$G\iso \cyclic{3^k}^4\sdpf(\cyclic{5}\sdp\cyclic{16})$.}
\end{enumerate}
\end{theoremD}

%

We saw that an insoluble Frobenius group with perfect
order classes is perfect with Frobenius complement isomorphic
to the $\{2,3,5\}$-group $\SL{2}{5}$.
It turns out that there are no Frobenius groups
with perfect order classes in which the Frobenius complement
is a soluble $\{2,3,5\}$-group.

\begin{theoremE}
A soluble Frobenius group whose Frobenius complement is a
$\{2,3,5\}$-group with order divisible by $30$ does not have perfect order classes.
\end{theoremE}

Our results suffice to explain all the Frobenius groups with perfect order classes
of order at most $15000$ in the library of Frobenius groups up to
that order in the computer algebra system \Maple{}.

As is common in studying groups with perfect order classes,
a considerable amount of number theory (all of an
elementary nature) makes its way into our arguments.

In Section~\ref{sec:prelim} we review some necessary background.
Our first main result, Theorem A, is proved in Section~\ref{sec:ordclassstruct}.
The proof of Theorem B occupies Section~\ref{sec:insoluble},
and Theorems C, D and E are proved in Section~\ref{sec:soluble}.
In Section~\ref{sec:conclusion} we suggest some problems for future work.

\section{Preliminaries}\label{sec:prelim} 

In this section, we describe our notation and terminology,
as well as some required background from number theory
and required results on groups with perfect order classes.

\subsection{A Bit of Number Theory}\label{ss:numtheory}

Throughout the paper, a \defn{prime} means a positive rational prime.
By a \defn{prime power} we mean a non-negative integer power
of a prime.
A prime power different from unity is said to be \defn{well-defined},
and a prime power that is composite is said to be \defn{proper}.
If $\pi$ is a set of primes,
then a positive integer is a \defn{$\pi$-number}
if every one of its prime divisors belongs to $\pi$.
For instance, $12$ is a $\{2,3,61\}$-number.
The greatest common divisor of integers $a$ and $b$ is
denoted by $\gcd( a, b )$.
We write $\eulerphi(n)$ for Euler's function of the
positive integer $n$.
Recall that, for a well-defined prime power $p^k$, we have
\begin{displaymath}
\eulerphi( p^k ) = p^{k-1}( p - 1 ),
\end{displaymath}
and that for relatively prime $a$ and $b$,
we have $\eulerphi(ab) = \eulerphi(a)\eulerphi(b)$.

We shall make extensive use of the following
result~\cites{Bang1886a,Bang1886b,BirkhoffVandiver1904,Zsigmondy1892}.

\begin{theorem}[Zsigmondy's Theorem]\label{thm:zsigmondy}
Let $a$ and $b$ be coprime positive integers such that $a > b$,
let $n$ be an integer greater than $1$,
and let $\epsilon = \pm1$.
If $\epsilon = 1$, assume that $(a,b,n) \neq (2,1,3)$ and,
if $\epsilon = -1$, assume that $(a,b,n) \neq (2,1,6)$ and that,
if $n = 2$, then $a + b$ is not a power of $2$.
Then $a^n + \epsilon b^n$ has a prime divisor $q$
such that $q$ does not divide $a^k + \epsilon b^k$,
for any positive integer $k$ less than $n$.
\end{theorem}

The prime $q$ in the statement of Zsigmondy's theorem is
called a \defn{primitive prime divisor} of $a^n + \epsilon b^n$.

Most often, Zsigmondy's theorem will be applied in the case
for which $b = 1$ and $\epsilon = -1$.
We can then interpret Zsigmondy's theorem as an assertion
that, apart from the indicated exceptions, any positive
integer $a$ can be taken to have any given order modulo a suitable prime.
Since, by Euler's theorem, we also have $a^{q-1} \equiv 1\pmod{q}$,
it follows that $q - 1$ is a multiple of $n$ and,
in particular, that $q > n$.

We shall also make frequent use of several well-known results
on consecutive prime powers.
The first says that $8$ and $9$ are the only consecutive
proper prime powers.
(This is much more elementary that the now proved~\cite{Mihailescu2004} Catalan conjecture.)
See, for example,~\cite{Ribenboim2000}*{Chapter 7}.

\begin{lemma}\label{lem:conspp}
Let $p$ and $q$ be prime numbers.
If $a$ and $b$ are integers greater than $1$ such that $p^a - q^b = 1$,
then $(p,q,a,b) = (3,2,2,3)$.
\end{lemma}

We shall also need to allow improper prime powers
in certain cases.

\begin{lemma}\label{lem:cons2a3b}
If $a$ and $b$ are non-negative integers such that
$2^a - 3^b = 1$, then
$(a,b) \in \{ (1,0), (2,1) \}$.
\end{lemma}

\begin{lemma}\label{lem:cons3b2a}
If $a$ and $b$ are non-negative integers such that $3^b - 2^a = 1$,
then $(a,b) \in \{ (1,1), (3,2) \}$.
\end{lemma}

\subsection{General Remarks on Groups}

All groups are supposed to be finite,
and will generally be written multiplicatively.
We use $1$ to denote the identity element of any group,
as well as to denote the trivial subgroup of a group.

If $g$ is an element of a group $G$, then $g^G$ denotes
the conjugacy class of $g$ in $G$, and $\order{g}$ is the
order of $g$.
In general, $\order{S}$ denotes the cardinality of a set $S$.

For a positive integer $n$, the cyclic group of order $n$
is written $\cyclic{n}$, while $\sym{n}$, $\alt{n}$ and
$\dih{n}$ denote, respectively, the symmetric group,
the alternating group, and the dihedral group of degree $n$.
Note that with this convention, we have $\order{\dih{n}} = 2n$.

If $n\geq 3$, then $\quat{n}$ denotes the generalised quaternion
group of order $2^n$, so that $\quat{3}$ is the ordinary
quaternion group of order $8$.
Recall that $Q_n$ has a presentation of the form
\begin{displaymath}
Q_n = \langle x,y \mid x^{2^{n-1}} = 1, y^2 = x^{2^{n-2}}, x^y = x^{-1}\rangle.
\end{displaymath}
It has an unique maximal cyclic subgroup $\langle x\rangle$
of index $2$.
Every element of $Q_n$ belonging to the non-trivial coset of
$\langle x\rangle$ in $Q_n$ has order $4$,
and $\langle x\rangle$ has just two elements of order $4$,
for a total of $2^{n-1} + 2$ elements of order $4$.
Elements of every other order belong to the maximal cyclic
subgroup $\langle x\rangle$, so we have the following result.

\begin{lemma}\label{lem:gqordcount}
Let $Q_n$ be a generalised quaternion group of order $2^n$,
where $n\geq 3$,
and for each positive integer $d$, let $\ordcount{d}{Q_n}$ denote
the number of elements of order $d$ in $Q_n$.
Then
\begin{displaymath}
\ordcount{2^k}{Q_n} = \begin{cases}
1,		& \text{for } k=0,1 \\
2^{n-1} + 2,	& \text{for } k = 2 \\
2^{k-1},	& \text{for } 3\leq k < n.
\end{cases}
\end{displaymath}
\end{lemma}

If $G$ is a group, and $r$ is a non-negative integer,
then $G^r$ denotes the direct product of $r$ copies of $G$,
with the convention that $G^0$ is a trivial group.
In particular, if $n$ is also a positive integer,
then $\cyclic{n}^r$ denotes the \defn{homocyclic group}
of rank $r$ and exponent $n$.
(In the sequel, $n$ will usually be a power of a prime.)
We remark that $A\times B$ is the direct product of
groups $A$ and $B$, while $A\sdp B$ is a semidirect product
of $A$ and $B$, where $A$ is normal in $A\sdp B$,
and $A$ and $B$ have trivial intersection.

For homocyclic $p$-groups, we use frequently the following result counting elements
of a given order.
\begin{lemma}\cite{FinchJones2003}*{Lemma 1}\label{lem:ordcounthomocyclic}
Let $p$ be a prime, and let $r$ and $k$ be positive integers.
For each integer $i$ with $1\leq i\leq k$,
the number of elements of order $p^i$ in a homocyclic
$p$-group of rank $r$ and exponent $p^k$ is given by
\begin{displaymath}
\ordcount{p^i}{\cyclic{p^k}^r} = p^{r(i-1)}(p^r - 1).
\end{displaymath}
\end{lemma}

If $p$ is a prime, $G$ is a $p$-group and $n$ is a non-negative integer,
then $\Omega_{(n)}(G) = \{ g\in G : g^{p^n} = 1 \}$,
and $\Omega_{n}(G) = \langle\Omega_{(n)}(G)\rangle$.
If $G$ is abelian, then $\Omega_{n}(G) = \Omega_{(n)}(G)$.

The derived subgroup of a group $G$ is denoted $[G,G]$,
and the derived quotient, or \defn{abelianisation} $G/[G,G]$
of $G$ is written $\abel{G}$.

A group $G$ is a \defn{Z-group} if every Sylow subgroup of $G$ is cyclic.
It is well-known that every Z-group is metacyclic
and has a presentation of the form
\begin{displaymath}
G = \langle x, y \mid x^{\alpha}, y^{\beta}, x^y = x^{\gamma} \rangle,
\end{displaymath}
where the positive integers $\alpha$, $\beta$ and $\gamma$ satisfy
$\gcd(\alpha, \beta) = 1 = \gcd(\alpha,\gamma - 1)$ and
$\gamma^\beta\equiv 1\pmod{\alpha}$.
Moreover, $\order{G} = \alpha\beta$ and $\langle x\rangle$ is
the derived subgroup of $G$.
If $\sigma$ is the order of $\gamma$ in the group $(\intmod{\alpha})^{\times}$
of units of the ring $\intmod{\alpha}$, then the centre of $G$ is
the subgroup $\langle y^{\sigma}\rangle$.

\subsection{Frobenius Groups}\label{ss:frobenius}

For background on the general theory of Frobenius groups
the reader may consult~\cites{Gorenstein1967,Passman2012,Robinson,Wolf2011}.
Proofs of everything not proved here may be found in those cited references.

We recall that a group $G$ is said to be a \defn{Frobenius group} if it has
a proper non-trivial subgroup $H$, called a \defn{Frobenius complement},
such that $H\cap H^g = 1$, for all $g\in G\setminus H$.
(We also say that $H$ is \defn{malnormal} in $G$.)
Then Frobenius' theorem asserts that the set
\begin{displaymath}
K = \{1\} \cup\bigcup_{g\in G} G\setminus H^g
\end{displaymath}
is a normal (indeed, characteristic) subgroup of $G$,
called the \defn{Frobenius kernel} of $G$.
It follows that $H$ acts faithfully and without fixed points on $K$.
We say that the action of $H$ on $K$ is ``Frobenius''.
Then $G$ is the semidirect product $K\sdp H$,
which we write as $K\sdpf H$ to indicate that the action
of $H$ on $K$ is Frobenius.
The Frobenius kernel of $G$ is uniquely determined
as the Fitting subgroup,
and a Frobenius complement is uniquely determined up to conjugacy.

The Frobenius kernel and complement have relatively prime
orders and, in fact, $\order{H}$ is a divisor of $\order{K} - 1$.
Every normal subgroup of $G$ either is contained in $K$ or contains it.
If $N$ is a non-trivial normal subgroup of $G$ contained in $K$,
then the subgroup $NH$ is a Frobenius group and,
if $N$ is a proper subgroup of $K$, then the quotient group $G/N$
is a Frobenius group with Frobenius kernel isomorphic to $K/N$
and Frobenius complement isomorphic to $H$.
The Sylow subgroups of every Frobenius complement are either cyclic
or generalised quaternion groups.
If $p$ and $q$ are (not necessarily distinct) primes,
then every subgroup of $H$ with order $pq$ is cyclic.
No subgroup of a Frobenius complement is a Frobenius group.
The Frobenius kernel $K$ is necessarily a nilpotent group and,
if the Frobenius complement $H$ has even order, then $K$ is,
in fact, abelian.

\subsection{Groups with Perfect Order Classes}\label{ss:poc}

This sub-section provides some background on groups with perfect order classes,
beginning with a precise definition.

Given a finite group $G$, we define an equivalence relation
on $G$ by identifying elements which have the same order.
The equivalence class of an element $g$ in $G$ with respect
to this equivalence relation is called the \defn{order class}
of $g$, and is denoted by $\ordclass{g}{G}$.
So we have
\begin{displaymath}
\ordclass{g}{G} = \{ x \in G : \order{x} = \order{g} \}.
\end{displaymath}
The cardinality of $\ordclass{g}{G}$ is denoted by $\ordcount{g}{G}$.
We say that $G$ has \defn{perfect order classes} if $\ordcount{g}{G}$
is a divisor of the order of $G$, for all elements $g$ in $G$.

For convenience, given a positive integer $n$,
we also denote by $\ordcount{n}{G}$ the cardinality of the order class
of elements of order $n$.
Note that $\ordcount{1}{G} = 1$, and this is a divisor of the
order of any group, so we usually do not mention the order class of $1$.

It is clear that $g^G \subseteq\ordclass{g}{G}$,
but in general the inclusion is proper.
However, each order class is a disjoint union
of one or more conjugacy classes.

\begin{example}
From Lemma~\ref{lem:gqordcount} it follows that no generalised
quaternion group $Q_n$ has perfect order classes,
since the number of elements of order $4$ has an odd divisor.
\end{example}

We shall use some basic results on order classes and
their cardinalities.

\begin{lemma}\cite{McCarron2020}\label{lem:normalHall}
If $H$ is a normal Hall subgroup of a group $G$,
then $H$ contains the complete order class of each
of its members.
\end{lemma}

\begin{lemma}\cite{McCarron2020}\label{lem:dirprod}
Let $A$ and $B$ be groups whose orders are relatively prime.
If $a$ and $b$ are positive integers such that $a$ divides $\order{A}$
and $b$ divides $\order{B}$, then
\begin{displaymath}
\ordcount{ab}{A\times B} = \ordcount{a}{A}\ordcount{b}{B}.
\end{displaymath}
\end{lemma}

The cyclic groups with perfect order classes have been completely
characterised by Das~\cite{Das2009b}.
This result follows from the elementary result that
a positive integer $n > 1$ is a multiple of $\eulerphi(n)$ precisely
when it is an even $\{2,3\}$-number.

\begin{proposition}\cite{Das2009b}*{Proposition 2.1}\label{prop:cyclic}
A cyclic group has perfect order classes if, and only if,
it is trivial or its order is an even $\{2,3\}$-number.
\end{proposition}

The next result is a key result in the study of groups with perfect order classes.

\begin{proposition}\cite{Das2009b}*{Proposition 2.2}\label{prop:eulerphi}
If $G$ is a finite group, and $g$ is an element of $G$,
then $\ordcount{g}{G}$ is divisible by $\phi(\order{g})$.
\end{proposition}

\begin{corollary}\cite{Das2009b}*{Corollary 2.3}\label{cor:pm1}
If $G$ is a finite group with perfect order classes then,
for each prime divisor $p$ of the order of $G$,
the order of $G$ is divisible by $p - 1$.
In particular, every non-trivial finite group with perfect order classes has even order.
\end{corollary}
It is also a consequence of Corollary~\ref{cor:pm1} that the smallest
odd prime divisor of a group with perfect order classes that is not
a $2$-group is a Fermat prime.

We shall need the another result, originally due to Das~\cite{Das2009b},
for which we provide a very short proof.

\begin{proposition}\cite{Das2009b}*{Proposition 2.4}\label{prop:twogroups}
A finite $2$-group has perfect order classes if, and only if, it is cyclic.
\end{proposition}
\begin{proof}
Every group of even order has an odd number of involutions so,
if a finite $2$-group is to have perfect order classes,
it must have an unique involution, which implies that it is either
cyclic or a generalised quaternion group.
But the number of elements of order $4$ in a generalised quaternion
group has an odd prime divisor, so a generalised quaternion group
cannot have perfect order classes.
\end{proof}

\subsection{Supporting Computations}\label{ss:comput}

Some of our arguments require the use of computer calculations to verify
the existence or non-existence of groups having particular properties.
All computations were performed using the computer algebra systems
\GAP{}\cite{GAP}, \Magma{}\cite{magma} and \Maple{}\cite{maplemanual} and,
when feasible, were replicated in two or more of these systems.

\section{The Structure of Order Classes in Frobenius Groups}\label{sec:ordclassstruct} 

To prove our main theorem, we must first paint a clear
picture of the structure of order classes in Frobenius groups.

\begin{lemma}\label{lem:main}
Let $G = K\sdpf H$ be a Frobenius group with Frobenius kernel
$K$ and Frobenius complement $H$.
\begin{enumerate}
\item{If $g \in K$, then $\ordclass{g}{G} = \ordclass{g}{K}$.}
\item{If $g \in G\setminus K$, then $\ordclass{g}{G}$ is the disjoint union
\begin{displaymath}
\ordclass{g}{G} = \bigcupdot_{x\in K} \ordclass{(h^x)}{H^x},
\end{displaymath}
where $h\in H\cap g^G$.
}
\end{enumerate}
\end{lemma}
\begin{proof}
Let $g$ be an element of $K$.
Since $K$ is a normal Hall subgroup of $G$,
we have from Lemma~\ref{lem:normalHall} that $\ordclass{g}{G}\subseteq K$.
Therefore,
\begin{displaymath}
\ordclass{g}{K} = K\cap\ordclass{g}{G} = \ordclass{g}{G}.
\end{displaymath}

To prove the second assertion,
fix an element $g$ of $G\setminus K$.
Since $G$ is a Frobenius group, $G\setminus K$ is the disjoint union
\begin{displaymath}
G\setminus K = \bigcupdot_{x\in K} H^x\setminus 1.
\end{displaymath}
Consequently, $H\cap g^G$ is non-empty,
so we may choose $h\in H\cap g^G$.
Then $\order{g} = \order{h} = \order{h^x}$,
so $\ordclass{g}{G} = \ordclass{(h^x)}{G}$,
for any $x\in K$.
Keeping in mind that $1\notin\ordclass{g}{G}$,
we therefore have for each $x\in K$ that
\begin{displaymath}
\ordclass{g}{G}\cap (H^x\setminus 1)
= \ordclass{(h^x)}{G}\cap (H^x\setminus 1)
= \ordclass{(h^x)}{G}\cap H^x
= \ordclass{(h^x)}{H^x}.
\end{displaymath}
Since $K\cap\ordclass{g}{G} = \emptyset$,
that is, $\ordclass{g}{G}\subseteq G\setminus K$,
we have
\begin{displaymath}
\ordclass{g}{G} = \ordclass{g}{G}\cap \left( \bigcupdot_{x\in K} H^x\setminus 1\right) = \bigcupdot_{x\in K}\ordclass{(h^x)}{H^x},
\end{displaymath}
as claimed.
\end{proof}

\begin{corollary}\label{cor:main}
Let $G = K\sdpf H$ be a Frobenius group as in Lemma~\ref{lem:main}.
\begin{enumerate}
\item{If $g\in K$, then $\ordcount{g}{G} = \ordcount{g}{K}$.}
\item{If $g\in G\setminus K$, then $\ordcount{g}{G} = \order{K}\ordcount{h}{H}$,
where $h\in H\cap g^G$.}
\end{enumerate}
\end{corollary}

With these preparations in hand,
we are now ready to prove our main result,
stated as Theorem~A in the Introduction.

\begin{theorem}\label{thm:main}
Let $G = K\sdpf H$ be a Frobenius group with Frobenius kernel $K$
and Frobenius complement $H$.
Then $G$ has perfect order classes if, and only if,
the following are true:
\begin{itemize}
\item[(a)]{$H$ has perfect order classes;}
\item[(b)]{$K$ is a homocyclic $p$-group for an odd prime number $p$; and,}
\item[(c)]{$\order{H} = p^r - 1$, where $r$ is the rank of $K$.}
\end{itemize}
\end{theorem}
\begin{proof}
Suppose that $H$ has perfect order classes and order equal
to $p^r - 1$, where $K$ is a homocyclic $p$-group of rank $r$.
The set $G\setminus 1$ is the disjoint union of $K\setminus 1$
and the sets $H^x\setminus 1$, for $x\in K\setminus 1$.
If $d$ is the order of some non-trivial element of $H$,
then $\ordcount{d}{H}$ divides $\order{H}$ because $H$ is
assumed to have perfect order classes.
Therefore, by Corollary~\ref{cor:main}, we have that
$\ordcount{d}{G} = \order{K}\ordcount{d}{H}$ divides
the order $\order{G} = \order{K}\order{H}$ of $G$.
Now, $K$ is homocyclic of rank $r$, so we can write
\begin{displaymath}
K\iso\cyclic{p^e}^r,
\end{displaymath}
where $p^e$ is the exponent of $K$.
Then every member of $K$ has order a divisor of $p^e$ and,
for each integer $i$ with $1\leq i\leq e$,
we have from Lemma~\ref{lem:ordcounthomocyclic} and condition (c) that
\begin{displaymath}
\ordcount{p^i}{G} = \ordcount{p^i}{K} = p^{r(i-1)}(p^r - 1) = p^{r(i-1)}\order{H},
\end{displaymath}
which is clearly a divisor of the order of $G$.
Therefore, any Frobenius group satisfying the conditions (a)-(c)
has perfect order classes.

For the converse, suppose that $G$ has perfect order classes.
To show that $H$ has perfect order classes,
let $m$ be the order of a non-trivial element of $H$.
Since $G$ has perfect order classes, it follows that
$\ordcount{m}{G}$ divides $\order{G}$, so we can write
\begin{displaymath}
\order{G} = s\ordcount{m}{G},
\end{displaymath}
for some positive integer $s$.
Now $H$ has exactly $\order{K}$ distinct conjugates $H^g$ in $G$,
for which we clearly have $\ordcount{m}{H^g} = \ordcount{m}{H}$.
Moreover, these conjugates have pairwise trivial intersection,
so that
\begin{displaymath}
\ordcount{m}{G} = \order{K}\ordcount{m}{H}.
\end{displaymath}
Therefore,
\begin{displaymath}
\order{K}\order{H} = \order{G} = s\ordcount{m}{G} = s\order{K}\ordcount{m}{H};
\end{displaymath}
whence,
\begin{displaymath}
\order{H} = s\ordcount{m}{H}.
\end{displaymath}
Therefore, $\ordcount{m}{H}$ divides $\order{H}$ and,
since $m$ was the order of an arbitrary member of $H$,
it follows that $H$ has perfect order classes.

Since $H$ has perfect order classes, and must be non-trivial,
Corollary~\ref{cor:pm1} tells us that the order of $H$ is even.
Therefore, the Frobenius kernel $K$ is an abelian group.
Suppose that the order of $K$ is divisible by distinct (odd) primes $p$ and $q$.
Then $K$ has elements of orders $p$, $q$ and $pq$,
every member of $G$ with order equal to any of these numbers belongs to $K$ and,
since $K$ is abelian, we have
$\ordcount{pq}{G} = \ordcount{p}{G}\ordcount{q}{G}$.
Because $K$ is abelian, we have also that $\order{x^G} = \order{H}$,
for each non-trivial element $x$ in $K$.
Since each non-trivial order class in $K$ is the disjoint union of conjugacy
$G$-classes, each of length $\order{H}$,
it follows that $\order{H}$ divides $\ordcount{x}{G}$,
for each $x\in K\setminus 1$.
Therefore, $\order{H}$ divides both $\ordcount{p}{G}$ and
$\ordcount{q}{G}$, so $\order{H}^2$ divides $\ordcount{pq}{G}$.
Since $\ordcount{pq}{G}$ divides $\order{G}$, this is impossible.
Therefore, $K$ is a $p$-group, for some (odd) prime number $p$.

If $K$ has rank $r$, then
\begin{displaymath}
\ordcount{p}{G} = \ordcount{p}{K} = \order{\Omega_{1}(K)} = p^r - 1.
\end{displaymath}
Since this must be a divisor of $\order{G}$,
and since $\gcd(p,p^r-1) = 1$, it follows that $p^r - 1$
divides $\order{H}$.
On the other hand, $\Omega_{1}(K)H$ is a Frobenius subgroup of $G$,
since $\Omega_{1}(K)$ is characteristic in $K$, hence, normal in $G$.
Therefore, the order of $H$ divides $p^r - 1 = \order{\Omega_{1}(K)} - 1$,
and we conclude that $\order{H} = p^r - 1$, as claimed.

Finally, suppose that $K$ is not homocyclic,
and write $K = I\times J$, where $I$ is the homocyclic
component of $K$ with exponent equal to the exponent $p^e$ of $K$,
and $J$ has exponent $p^d$ with $d < e$.
Then $\Omega_{d}(K)$ is again a characteristic subgroup of $K$,
and is proper since $d < e$,
so $G/\Omega_{d}(K)$ is a Frobenius group with Frobenius complement
isomorphic to $H$ and Frobenius kernel isomorphic to $K/\Omega_{d}(K)$.
But $K/\Omega_{d}(K)$ has rank $s$ less than the rank $r$ of $K$,
and $\order{H} = p^r - 1$ must be a divisor of $p^s - 1$,
which is impossible.
Therefore, $K$ must be homocyclic.

This completes the proof.
\end{proof}

\section{Insoluble Frobenius Groups}\label{sec:insoluble} 

In this section, we give a precise description of the Frobenius
groups with perfect order classes that are insoluble.

Let us begin with an example to show that insoluble Frobenius
groups with perfect order classes exist.

\begin{example}
The smallest insoluble Frobenius group $G$ has order $14520$,
and has a Frobenius complement isomorphic to $\SL{2}{5}$,
while the Frobenius kernel is the elementary abelian
group $\cyclic{11}^2$.
It is a perfect group ($G = [G,G]$),
and has the order class profile
\begin{center}
\begin{tabular}{r|cccccccc}\hline
$d$               & $1$ &  $2$  &   $3$  &   $4$  &   $5$  &   $6$  &  $10$  & $11$ \\
$\ordcount{d}{G}$ & $1$ & $121$ & $2420$ & $3630$ & $2904$ & $2420$ & $2904$ & $120$ \\\hline
\end{tabular}
\end{center}
A simple numerical check shows that $G$ has perfect order classes.
Alternatively, one can check that $\SL{2}{5}$ has perfect order
classes, and that $\order{\SL{2}{5}} = 120 = 11^2 - 1$,
while the rank of the Frobenius kernel is $2$, so the fact
that $G$ has perfect order classes follows from Theorem~\ref{thm:main}.
\end{example}

We are going to show that this example is essentially typical
of insoluble Frobenius groups with perfect order classes.
To this end, we shall need to solve a particular Diophantine
equation involving prime powers.
The purely number-theoretic proof of the following result we
defer until the end of this section.

\begin{proposition}\label{prop:dio240}
Let $p$ and $q$ be prime numbers, both greater than $5$,
and let $r$ and $m$ be positive integers with $r > 1$.
Suppose that $q - 1$ is a divisor of $240$, but that $q$ does
not divide $240$.
Then the equation
\begin{displaymath}
p^r - 1 = 240q^m
\end{displaymath}
has exactly two solutions:
\begin{displaymath}
(p,r,q,m) \in \{ (11, 4, 61, 1), (41, 2, 7, 1) \}.
\end{displaymath}
\end{proposition}

\begin{lemma}\label{lem:insolmustbeperfect}
An insoluble Frobenius group with perfect order classes is perfect.
\end{lemma}
\begin{proof}
Let $G$ be a Frobenius group with perfect order classes and assume
that $G$ is insoluble but not perfect.
Then by Theorem~\ref{thm:main}, the Frobenius kernel $K$ of $G$
is a homocyclic $p$-group of the form $\cyclic{p^k}^r$,
for some odd prime $p$, and positive integers $k$ and $r$,
a Frobenius complement $H$ of $G$ has perfect order classes,
and we have
\begin{displaymath}
\order{H} = p^r - 1.
\end{displaymath}
Since $H$ is not abelian, it follows that the rank $r$ of $K$ is
greater than $1$ since, otherwise, $H$ could not be isomorphic to
a subgroup of $\aut{K}$.

Since $H$ is insoluble, there is a subgroup $L$, of index at
most $2$ in $H$, such that $L = S\times M$, where $S$ is
isomorphic to $\SL{2}{5}$, and $M$ is a Z-group whose order
is coprime to $30$.
Then we have
\begin{displaymath}
\order{H} = 120\order{M}[H:L].
\end{displaymath}

Suppose first that the subgroup $M$ is trivial.
Then $[H:L] = 2$, since $H$ is not perfect,
so $p^r = 241$, whence $p = 241$ and $r = 1$,
contrary to the fact that $r > 1$.
Therefore, it must be that $M$ is non-trivial.

Let $q$ be a prime divisor of the order of $M$.
Since $\gcd(\order{M},30)=1$, hence, $q > 5$.
Then $L = S\times M$ has an element of order $5q$,
and every element of $H$ with order $5q$ belongs to $L$,
since each member of $H\setminus L$ has even order if
it is non-empty.
Then
\begin{displaymath}
\ordcount{5q}{H} = \ordcount{5q}{L} = \ordcount{5}{S}\ordcount{q}{M} = 24(q-1).
\end{displaymath}
Since $q - 1$ is an even number, it follows that $\ordcount{5q}{H}$
is divisible by $16$.
Thus, $[H:L] = 2$ because, otherwise, the order of $H$ would not
be divisible by $16$.

Now suppose that $s$ is another prime divisor of the order of $M$.
Then $s > 5$ is also odd and,
if $M$ has an element of order $qs$, then every member of $H$ whose
order is equal to $qs$ belongs to $L$,
and so $\ordcount{5qs}{H} = \ordcount{5qs}{L}$ is a divisor of
the order of $H$.
But
\begin{displaymath}
\ordcount{5qs}{H} = \ordcount{5}{S}\ordcount{qs}{M} = 24(q-1)(s-1),
\end{displaymath}
which is divisible by $32$ since $q$ and $s$ are odd.
This is not possible, since $\ordcount{5qs}{H}$ must divide the
order of $H$.
In particular, if $M$ is cyclic, then $M$ is a $q$-group,
for some prime $q$ with $q > 5$.

If $M$ is not cyclic, then $M$ is a Z-group and we can write
\begin{displaymath}
M = \langle x, y \mid x^{\alpha}, y^{\beta}, x^y = x^{\gamma} \rangle,
\end{displaymath}
where $\alpha$, $\beta$ and $\gamma$ are positive integers such that
$\gcd(\alpha,\beta) = 1 = \gcd(\alpha, \gamma - 1)$.
If either of $\alpha$ and $\beta$ is not a prime power,
then $M$ has an element whose order is a product of two distinct
primes, which we have just argued is not possible.
Therefore, $\alpha$ and $\beta$ are powers of different primes.
We claim that the action of $\langle y\rangle$ on $\langle x\rangle$
is fixed-point free.
For, if not, then some non-trivial power of $x$ commutes with some
non-trivial power of $y$, and so their product has a power whose
order is a product of two distinct primes.
Therefore, $M$ is a Frobenius group, with Frobenius kernel
$\langle x\rangle$ and Frobenius complement $\langle y\rangle$.
But $M$ is a subgroup of the Frobenius complement $H$ of $G$,
which cannot contain a Frobenius subgroup.
Therefore $M$ must be a cyclic $q$-group.
Let $\order{M} = q^m$, where $m$ is a positive integer,
so that
\begin{displaymath}
\order{H} = 240q^m.
\end{displaymath}

We have the relation
\begin{displaymath}
p^r - 1 = 240q^m,
\end{displaymath}
where $q-1$ is a divisor of $240$,
and Proposition~\ref{prop:dio240} implies that the only solutions are
\begin{displaymath}
(p,r,q,m) \in \{ (11, 4, 61, 1 ), ( 41, 2, 7, 1 ) \}.
\end{displaymath}
In the first case, $L \iso\SL{2}{5}\times\cyclic{61}$
(and $K\iso\cyclic{11}^4$), whilst in the second case,
we have $L\iso\SL{2}{5}\times\cyclic{7}$ (and $K\iso\cyclic{41}^2$).

If $L\iso\SL{2}{5}\times\cyclic{7}$, then $L$ has an element
of order $35$ and every element of $H$ with order equal to $35$
belongs to $L$.
Hence,
\begin{displaymath}
\ordcount{35}{H} = \ordcount{35}{L} = \ordcount{5}{S}\ordcount{7}{M} = 24\cdot 6 = 2^4\cdot \mathbf{3^2},
\end{displaymath}
which cannot divide $\order{H} = 2^4\cdot \mathbf{3}\cdot 5\cdot 7$.

If $L\iso\SL{2}{5}\times\cyclic{61}$, then $L$ has an element
of order $305$ and, as in the previous case, we have
\begin{displaymath}
\ordcount{305}{H} = \ordcount{5}{S}\ordcount{61}{M} = 24\cdot 60 = \mathbf{2^5\cdot 3^2}\cdot 5
\end{displaymath}
which, again, cannot divide $\order{H} = \mathbf{2^4\cdot 3}\cdot 5\cdot 61$.

This final contradiction forces us to abandon our hypothesis
that $G$ is not perfect, completing the proof.
\end{proof}


If $G = K\sdpf H$ is a Frobenius group with Frobenius kernel $K$
and Frobenius complement $H$, then $[G,G] = K[H,H]$.
Consequently, $G$ is perfect if, and only if, $H$ is perfect,
and a well-known result of Zassenhaus states that a Frobenius
complement is perfect if, and only if, it is isomorphic to
$\SL{2}{5}$.
It follows that the only special linear groups that can occur
as Frobenius complements are $\SL{2}{5}$ and $\SL{2}{3}$.
In fact, $\SL{2}{2}\iso\sym{3}$ is itself a Frobenius group,
so it cannot be a Frobenius complement.
And, for prime powers $q > 5$, the group $\SL{2}{q}$ is perfect,
as is the simple group $\SL{2}{4}\iso\alt{5}$.

Our next result provides a complete characterisation of the
insoluble Frobenius groups with perfect order classes.

\begin{theorem}\label{thm:perfect}
Let $G = K\sdpf H$ be an insoluble Frobenius group,
with Frobenius kernel $K$ and Frobenius complement $H$.
Then $G$ has perfect order classes if, and only if,
$K\iso\cyclic{11^k}^2$, for some positive integer $k$,
and $H\iso\SL{2}{5}$.
\end{theorem}
\begin{proof}
According to Lemma~\ref{lem:insolmustbeperfect},
an insoluble Frobenius group with perfect order classes must be perfect.
Let $G = K\sdpf H$ be a perfect Frobenius group with perfect order classes.
Then $H\iso\SL{2}{5}$ and $K$ is a homocyclic $p$-group of the form
$\cyclic{p^k}^r$, for an odd prime $p$ and positive integers $k$ and $r$.
We also have $p^r - 1 = \order{H} = 120$, whence, $p^r = 121 = 11^2$,
so $p = 11$ and $r = 2$.
Therefore, $K\iso\cyclic{11^k}^2$, as claimed.

For the converse, assume that $H\iso\SL{2}{5}$ and that $K$
is a homocyclic $11$-group of rank $2$,
so that $K\iso\cyclic{11^k}^2$, for some positive integer $k$.
Then $11^2 - 1 = \order{H}$, so $G$ has perfect order classes,
by Theorem~\ref{thm:main}.
\end{proof}


\begin{example}
There are perfect Frobenius groups,
$\cyclic{19}^2\sdpf\SL{2}{5}$ of order $43320$,
$\cyclic{29}^2\sdpf\SL{2}{5}$ of order $100920$, and,
$\cyclic{59}^2\sdpf\SL{2}{5}$ of order $417720$,
but none of these has perfect order classes.
\end{example}

Let us now turn to our proof of Proposition~\ref{prop:dio240}.

\begin{proof}[Proof of Proposition~\ref{prop:dio240}]
We are solving the equation
\begin{displaymath}
p^r - 1 = 240q^m,
\end{displaymath}
where $r$ is an integer greater than $1$,
and $p$ and $q$ are odd primes such that $q - 1$ divides $240$,
but $q$ does not divide $240$.

The prime numbers $q$ such that $q - 1$ is a divisor of $240$,
but $q$ does not divide $240$ are
\begin{displaymath}
q\in \{ 7, 11, 13, 17, 31, 41, 61, 241 \}.
\end{displaymath}
That the stated solutions are valid is verified by
direct numerical calculation:
\begin{displaymath}
11^4 = 1 + 240\cdot 61^1, 41^2 = 1 + 240\cdot 7^1.
\end{displaymath}

We show that these are the only solutions.
The method of solution is to factor the expression $p^r - 1$,
and then to determine how the prime divisors of $240q^m$ may
be allocated to each factor.
To this end, we consider two primary cases,
according to whether $r$ is even or odd.
In each case, there are a number of sub-cases,

First suppose that $r$ is even, say $r = 2s$,
where $s$ is a positive integer.
Then $p^r - 1$ is a difference of squares, so we have
\begin{displaymath}
2^4\cdot 3\cdot 5\cdot q^m = p^{2s} - 1 = (p^s - 1)(p^s + 1).
\end{displaymath}
Since the even integers $p^s - 1$ and $p^s + 1$ differ by $2$,
exactly one of them is divisible by $4$, hence, it is divisible by $8$,
while the other is twice an odd divisor of $240q^m$.
These odd divisors are $3$, $5$, $15$ as well as multiples of those
three numbers by powers of $q$.
Furthermore, $\gcd( p^s - 1, p^s + 1 ) = 2$, so that they have no
odd prime divisors in common.
In particular, each of $3$, $5$ and $q$ divides exactly one of
$p^s - 1$ and $p^s + 1$.
Therefore, if $q$ divides one of these two factors,
then $q^m$ divides that same factor.

Suppose that $p^s + 1$ is exactly divisible by $2$,
so that $p^s - 1$ is divisible by $8$, and $p^s + 1$
is twice an odd divisor of $240q^m$.
We consider the possibilities in turn;
two of these lead to the expected solutions,
while the others result in absurdities.

%
%
%
%

If $p^s + 1$ is divisible by $q$, then it is divisible by $q^m$.
If $p^s + 1 = 2\cdot q^m$, then $p^s - 1 = 2^3\cdot 3\cdot 5 = 120$,
so $p^s = 121 = 11^2$, whence, $p = 11$ and $s = 2$.
(Therefore, $r = 2s = 4$.)
Then $2q^m = p^s + 1 = 122$, so that $q^m = 61$, yielding $q = 61$ and $m = 1$.
This produces one of our expected solutions $11^4 = 1 + 240\cdot 61$.

If $p^s + 1 = 2\cdot 3\cdot q^m$, then $p^s - 1 = 8\cdot 5 = 40$,
so that $p^s = 41$ and hence $p = 41$ and $s = 1$ (so $r = 2$).
Then $p^s + 1 = 42$, so $q^m = 7$, giving $q = 7$ and $m = 1$.
This produces our second expected solution $41^2 = 1 + 240\cdot 7$.

The remaining sub-cases in which
$p^s + 1 \in \{6, 10, 30, 10q^m, 30q^m\}$
all lead quickly either to arithmetic absurdities
or to a contradiction of our stated hypotheses.

We must also consider all the possibilities for which $p^s - 1$
is equal to twice an odd divisor of $240q^m$.
However, each of these also result in contradictions by
elementary considerations.

%
%
%
%
%
%
%

This concludes our analysis for the case in which $r$ is even,
so assume now that $r$ is odd.
We examine a number of sub-cases, each leading to a contradiction.

We no longer have the convenient factorisation afforded by
a difference of squares, but we can pull out a factor of $p-1$ and write
\begin{displaymath}
2^4\cdot 3\cdot 5\cdot q^m = p^r - 1 = (p-1)(1 + p + \cdots + p^{r-1}).
\end{displaymath}
Since $r$ is odd, hence, $1 + p + \cdots + p^{r-1}$ is odd,
so $16 = 2^4\mid p-1$.
Therefore, $p - 1$ is a divisor of $240q^m$ divisible by $16$.
We cannot have either $p - 1 = 16\cdot 3$ or $p - 1 = 16\cdot 5$ since,
in either case, $p$ is not a prime.

Suppose that $p - 1 = 16$, so that $p = 17$, and
\begin{displaymath}
1 + p + \cdots + p^{r-1} = 15q^m.
\end{displaymath}
Reducing this equation modulo $15$, we obtain
\begin{displaymath}
1 + 2 + \cdots + 2^{r-1} \equiv 0\pmod{15}.
\end{displaymath}
Noting that $1 + 2 + 4 + 8 = 15$, and taking four terms at a time
on the left hand side of this equation, it follows that $r$ is
divisible by $4$, contrary to our assumption that $r$ is odd.

If $p - 1 = 240$, then $p = 241$ and we have
\begin{displaymath}
241^r - 1 = 240q^m,
\end{displaymath}
so that $241^r\equiv 1\pmod{240q^m}$.
In particular, $241^r\equiv 1\pmod{q}$.
For $q \in \{ 7, 11, 13, 17, 31, 41 \}$, we have
that $241$ has even order modulo $q$, which implies that $r$ is even.
Thus, the only possibility for $q$ is $q = 61$, since $241$ has order
equal to $5$ modulo $61$, which entails that $r$ is a multiple of $5$.
Write $r = 5s$.
If $s = 1$ so that $r = 5$, then we have
\begin{displaymath}
241^r - 1 = 241^5 - 1 = 812990017200 = 2^4\cdot 3\cdot 5^2\cdot \mathbf{61}\cdot 11106421.
\end{displaymath}
Hence, $(241^5 -1)/240 = 5\cdot 61\cdot 11106421$, which is not a power of $61$,
so $s > 1$ and $r > 5$.
But notice that $241^5 - 1$ is divisible by $61$,
as well as by $2$, $3$ and $5$.
Zsigmondy's theorem then implies that $241^r - 1$ has a prime divisor that
does not divide $241^5 - 1$, since $r > 5$.
This is impossible, so we conclude that $p\neq 241$.

If $q^m$ divides $p - 1$ then, since $16$ also divides $p - 1$,
it follows that $1 + p + \cdots + p^{r-1}$ is a divisor of $15$,
which is clearly impossible for any prime $p > 5$.

We are left with the case in which $q$ is a divisor of both
$p - 1$ and $1 + p + \cdots + p^{r-1}$.
We show that this case also cannot occur.
Now, $3$ must divide one of $p - 1$ and $p + 1$, so $3$ is a divisor
of $p^2 - 1$.
By Euler's Theorem, we also have $p^4\equiv 1\pmod{5}$,
so that $5$ is a divisor of $p^4 - 1$.
Since $p\equiv 1\pmod{q}$ and $q$ divides $1 + p + \cdots + p^{r-1}$,
we have
\begin{displaymath}
0\equiv 1 + p + \cdots + p^{r-1}\equiv r\pmod{q};
\end{displaymath}
whence $q$ divides $r$ and, in particular, $r\geq q > 5$.
Therefore, $p^r - 1$ has a primitive prime divisor $s$,
for which $s\mid p^r - 1$ but $s\nmid p^i - 1$, for $1\leq i < r$.
But we have just seen that each of $2$, $3$, $5$ and $q$ occurs as
a divisor of $p^i - 1$, for some $i < r$.
Therefore, $q$ cannot divide both $p - 1$ and $1 + p + \cdots + p^{r-1}$.

This shows that $p^r - 1 = 240\cdot q^m$ has only the two solutions
claimed.
\end{proof}

\section{Soluble Frobenius Groups}\label{sec:soluble} 

We now turn to our study of the much richer class of soluble Frobenius groups.
We are able to characterise those with perfect order classes in
a number of important cases.
Recall that the Sylow subgroups of a Frobenius complement are
either cyclic or generalised quaternion.
Therefore, a Frobenius complement is abelian if, and only if,
it is cyclic.
We therefore consider first the Frobenius groups with cyclic Frobenius
complement and provide a complete characterisation of those that
have perfect order classes.
Next, we consider those Frobenius groups whose Frobenius complement
is nilpotent, and show that no new examples of groups with perfect
order classes arise other than those with cyclic complement.
Next, we turn to Frobenius groups whose complements have order divisible
by just one odd prime; these are necessarily soluble, by Burnside's theorem,
and we can assume that the Frobenius complement is not nilpotent,
in light of previous results.
Finally, since we have seen that a Frobenius complement in an insoluble
Frobenius group with perfect order classes is isomorphic to the
$\{2,3,5\}$-group $\SL{2}{5}$, we show that there are no soluble examples
whose Frobenius complement is a $\{2,3,5\}$-group.

\subsection{Frobenius Groups with Cyclic Complement}\label{ssec:cyccompl}

We now examine Frobenius groups whose Frobenius complement is cyclic.
(Recall that a Frobenius complement is abelian if, and only if,
it is cyclic.)

We shall need the following simple lemma stating that the only
pairs of consecutive $\{2,3\}$-numbers are
$(1,2)$, $(2,3)$, $(3,4)$ and $(8,9)$.
The reader should have no difficulty supplying a proof using
Lemma~\ref{lem:cons2a3b} and Lemma~\ref{lem:cons3b2a}.
\begin{lemma}\label{lem:cons23nums}
If $x,y,u,v$ are non-negative integers such that
$2^x3^y - 2^u3^v = 1$, then
\begin{displaymath}
(x,y,u,v)\in \{ (1,0,0,0), (0,1,1,0), (2,0,0,1), (0,2,3,0) \}.
\end{displaymath}
\end{lemma}
%
%

\begin{corollary}\label{cor:cons23nums}
Let $p$ be a positive integer such that both $p - 1$
and $p + 1$ are $\{2,3\}$-numbers.
Then $p\in \{ 2, 3, 5, 7, 17\}$.
\end{corollary}
\begin{proof}
Since $p + 1$ and $p - 1$ are both $\{2,3\}$-numbers,
we can write $p + 1 = 2^a3^b$ and $p - 1 = 2^c3^d$,
for suitable non-negative integers $a$, $b$, $c$ and $d$.
Then $2^a3^b - 1 = p = 2^c3^d + 1$,
so we have
\begin{displaymath}
2^c3^d = 2^a3^b - 2.
\end{displaymath}

First suppose that $a = 0$, so we have $2^c3^d = 3^b - 2$,
or $3^b = 2^c3^d + 2$.
If $c > 0$, this yields $3^b = 2(2^{c-1}3^d + 1)$,
which implies that $3^b$ is even, so it must be that
$c = 0$ also.
This yields $3^d = 3^b - 2$ or $3^b - 3^d = 2$,
which clearly forces $b = 1$ and $d = 0$.
Thus, we have $(a,b,c,d) = (0,1,0,0)$,
from which we obtain $p = 2$.

Now assume that $a > 0$.
Then
\begin{displaymath}
2^c3^d = 2(2^{a-1}3^b - 1).
\end{displaymath}
If $c = 0$, this again implies that $3^d$ is even,
so $c$ must also be positive.
Then $2^{c-1}3^d = 2^{a-1}3^b - 1$, or
\begin{displaymath}
2^{a-1}3^b - 2^{c-1}3^d = 1.
\end{displaymath}
From Lemma~\ref{lem:cons23nums} we have
\begin{displaymath}
(a,b,c,d) \in \{ (1,1,2,0), (1,2,4,0), (2,0,1,0), (3,0,1,1) \}.
\end{displaymath}
These solutions yield (using either $p = 1 + 2^c3^d$ or $p = 2^a3^b - 1$)
$p \in \{5, 17, 3, 7 \}$.
This completes the proof.
\end{proof}

We now are ready to analyse the Frobenius groups with perfect order classes and
cyclic Frobenius complement.

\begin{proposition}\label{prop:cycliccompl}
Let $G = K\sdpf H$ be a Frobenius group with cyclic Frobenius complement $H$.
Then $G$ has perfect order classes if, and only if,
one of the following is true:
\begin{enumerate}
\item{$K$ is a cyclic $p$-group, for some Pierpont prime $p = 1 + \order{H} > 2$,
and $H$ is a $\{2,3\}$-group of even order; or}
\item{$G$ is isomorphic to one of the following groups:
\begin{enumerate}
\item{$(\cyclic{3^k}\times\cyclic{3^k})\sdpf\cyclic{8}$;}
\item{$(\cyclic{5^k}\times\cyclic{5^k})\sdpf\cyclic{24}$;}
\item{$(\cyclic{7^k}\times\cyclic{7^k})\sdpf\cyclic{48}$;}
\item{$(\cyclic{17^k}\times\cyclic{17^k})\sdpf\cyclic{288}$,}
\end{enumerate}
for some positive integer $k$.}
\end{enumerate}
\end{proposition}

Note that in the cyclic case,
the possibility that the Frobenius complement is a $2$-group
is not excluded.

\begin{proof}
Suppose that $G$ has perfect order classes.
Then $H$ has perfect order classes, and there is an odd prime $p$
such that $K$ is a homocyclic group of the form $\cyclic{p^k}^r$,
for some positive integers $k$ and $r$,
and we have $\order{H} = p^r - 1$.
Since $H$ is cyclic with perfect order classes,
it follows from Proposition~\ref{prop:cyclic}
that $H$ is a $\{2,3\}$-group of even order,
so we may write
\begin{displaymath}
\order{H} = 2^{\alpha}3^{\beta},
\end{displaymath}
where $\alpha$ is a positive integer and $\beta$ is a non-negative integer.

If $r = 1$, then $K$ is cyclic and $p = 1 + \order{H}$ is a Pierpont prime,
so assume that $r > 1$.

If $r$ is odd, then $r\geq 3$.
If $H$ is a $2$-group, whence $\beta = 0$, then $p^r - 1$ is a power of $2$.
If $p^r - 1 = 2$, then $p^r = 3$, so $p = 3$ and $r = 1$,
contrary to the assumption that $r > 1$.
Therefore, $p^r - 1$ must be a proper power of $2$,
say, $p^r - 1 = 2^m$, where the integer $m$ is greater than $1$.
According to Lemma~\ref{lem:conspp},
this implies that $p = 3$, $r = 2$ and $m = 3$,
this time contradicting the assumption that $r$ is odd.
Thus, in this case, $H$ cannot be a $2$-group,
and $3$ is a divisor of $\order{H} = p^r - 1$.
Since $p$ and the order of $H$ must be relatively prime,
we see that $p > 3$, and hence, $3\mid p^2 - 1$.
Certainly, $2$ divides $p - 1$ (and so also $p^2-1$).
Since $r\geq 3$, Zsigmondy's theorem implies that $p^r - 1$
has a prime divisor $q$ such that $q$ does not divide $p^i - 1$,
for any integer $i$ with $1\leq i < r$.
In particular, $q$ does not divide $p^2 - 1$, so $q > 3$.
This is in contradiction to the hypothesis that $p^r - 1 = \order{H}$
is a $\{2,3\}$-number.
We conclude, therefore, that $r$ is odd only if $r = 1$.

Suppose now that $r = 2s$ is even,
where $s$ is a positive integer.
If $H$ is a $2$-group, then $\order{H} = 2^{\alpha}$, and we have
\begin{displaymath}
2^{\alpha} = p^{2s} - 1 = (p^s - 1)(p^s + 1) = (p^2 - 1)(1 + p^2 + p^4 + \cdots + p^{2(s-1)}).
\end{displaymath}
This implies that $p^2 - 1$ and $p^s - 1$ are powers of $2$.
Since $p^2 - 1 \neq 2$, for any odd prime $p$, therefore,
$p^2 - 1$ is a proper power of $2$,
so $p^2 - 1  = 2^3$, and $p = 3$,
thanks to Lemma~\ref{lem:conspp}.
Since $p^s - 1 = 3^s - 1$ is also a power of $2$,
we must have $s = 1$ and $r = 2$.
Then, from $\order{H} = p^r -  1$, we see that $\order{H} = 8$.
Therefore,
\begin{displaymath}
G \iso \cyclic{3^k}^2 \sdpf\cyclic{8},
\end{displaymath}
giving the first of our four cases with a Frobenius kernel of rank $2$.

If $H$ is not a $2$-group, then $\beta > 0$ and $p > 3$.
Then, from
\begin{displaymath}
2^{\alpha}3^{\beta} = p^{2s} - 1 = (p^s - 1)(p^s + 1),
\end{displaymath}
we see that both $p^s - 1$ and $p^s + 1$ are $\{2,3\}$-numbers,
so from Corollary~\ref{cor:cons23nums},
we have $p^s  \in \{2,3,5,7,17\}$, whence $s = 1$ and $p\in\{5,7,17\}$.

For $p = 5$, we have $H\iso\cyclic{24}$, and so
\begin{displaymath}
G\iso \cyclic{5^k}^2 \sdpf\cyclic{24}.
\end{displaymath}
For $p = 7$, we have $H\iso\cyclic{48}$, giving
\begin{displaymath}
G \iso \cyclic{7^k}^2 \sdpf\cyclic{48}.
\end{displaymath}
Finally, for $p = 17$ we have $H\iso\cyclic{288}$, so that
\begin{displaymath}
G \iso \cyclic{17^k}^2 \sdpf\cyclic{288}.
\end{displaymath}

For the converse, we consider each case in turn,
beginning with the case of a cyclic Frobenius kernel.
Suppose then that $K$ is a cyclic $p$-group,
where $p = 1 + \order{H}$ is a Pierpont prime and $p > 2$.
Then $\order{H}\geq 2$ is a $\{2,3\}$-number and,
since $\order{H} + 1$ is an odd prime,
it follows that $\order{H}$ is even.
Therefore, $H$ has perfect order classes by Proposition~\ref{prop:cyclic}.
Since $K$ is cyclic, it has rank equal to $1$,
so the relation $\order{H} = p - 1$ is satisfied.
Therefore, by Theorem~\ref{thm:main}, $G$ has perfect order classes.

For the rank $2$ cases, we have
$G\iso \cyclic{p^k}^2\sdpf\cyclic{n}$,
where
\begin{displaymath}
(p,n)\in \{ (3,8), (5,24), (7,48), (17,288) \},
\end{displaymath}
and $k$ is an arbitrary positive integer.
In each case $\cyclic{n}$ is a $\{2,3\}$-group of even order and,
as such, has perfect order classes, according to Proposition~\ref{prop:cyclic}.
The Frobenius kernels $\cyclic{p^k}^2$ are homocyclic,
and it is easily verified numerically that,
in each case, we have $n = p^2 - 1$.
Now invoking Theorem~\ref{thm:main} suffices to complete the proof.
\end{proof}

The following is a slight generalisation of \cite{TuanHai2010}*{Theorem 2.1},
which asserts that a dihedral group has perfect order class if, and only if,
its degree is a power of $3$.

\begin{corollary}\label{cor:c2compl}
A Frobenius group whose Frobenius complement is a group of order $2$ has
perfect order classes if, and only if, it is a dihedral group of degree
equal to a power of $3$.
\end{corollary}

\begin{example}\label{ex:c4compl}
Let $G = K\sdpf H$ be a Frobenius group with Frobenius complement
$H$ a cyclic group of order $4$.
Then $G$ has perfect order class if, and only if,
$K$ is a non-trivial cyclic $5$-group.
\end{example}

\begin{example}\label{ex:pholomorph}
If $p$ is a prime number, then $\Hol(\cyclic{p})$ has perfect order
classes if, and only if, $p$ is an odd Pierpont prime.
\end{example}

\subsection{Frobenius Groups with Nilpotent Complement}\label{ss:nilpcompl} 

In this section we aim to show that a nilpotent Frobenius complement in a Frobenius
group with perfect order classes is necessarily cyclic.
Recall that a Sylow $2$-subgroup of a Frobenius complement is either
cyclic or generalised quaternion,
and that the odd order Sylow subgroups of any Frobenius complement
are cyclic.
Because a nilpotent Frobenius complement is the direct product of
its Sylow subgroups, it can be non-abelian only if it is
a direct product of a generalised quaternion group and a cyclic group
of odd order.

Let us begin with an example demonstrating the existence of
Frobenius complements with perfect order classes that are nilpotent
and non-abelian.

\begin{example}\label{ex:q4c5k}
Let $H \iso Q_{4}\times\cyclic{5^k}$, for a positive integer $k$,
where $Q_{4}$ is a generalised quaternion group of order $16$.
Then $H$ is a non-abelian nilpotent Frobenius complement with perfect order classes.
To see this, we tabulate the number of elements of each possible
order $m$ in $H$, where $i$ is any integer with $1\leq i\leq k$.
\begin{center}
\begin{tabular}{c|cccccccc}
\hline
$m$			& $1$ & $2$ & $4$  & $8$ &  $5^i$       &  $2\cdot 5^i$    & $4\cdot 5^i$ & $8\cdot 5^i$ \\
$\ordcount{m}{H}$	& $1$ & $1$ & $10$ & $4$ & $4\cdot 5^i$ & $4\cdot 5^{i-1}$ & $2\cdot 5^i$ & $16\cdot 5^i$ \\\hline
\end{tabular}
\end{center}
In each case, we see that $\ordcount{m}{H}$ is a divisor of the
order $\order{H} = 16\cdot 5^k$.
\end{example}

Our next result shows that this example is one of just two
families of non-abelian nilpotent Frobenius complements that
have perfect order classes.

\begin{theorem}\label{thm:nilpcompl}
Let $H$ be a non-abelian nilpotent Frobenius complement.
Then $H$ has perfect order classes if, and only if,
$H$ is isomorphic either to $Q_{3}\times\cyclic{3^k}$
or to $Q_{4}\times\cyclic{5^k}$,
for some positive integer $k$.
\end{theorem}
\begin{proof}
Assume that $H$ has perfect order classes;
hence, the order of $H$ is even.
Since $H$ is nilpotent, it is the direct product of its
Sylow subgroups.
And, since $H$ is a Frobenius complement, its Sylow subgroups
are cyclic or generalised quaternion.
Because $H$ is non-abelian, its Sylow $2$-subgroup $Q$ must be
a generalised quaternion group, say $Q\iso Q_{n}$, where $n\geq 3$.
Therefore,
\begin{displaymath}
H = Q\times T,
\end{displaymath}
where $T$ is a cyclic group of odd order (the direct product of
the odd order Sylow subgroups of $H$).
Since generalised quaternion groups do not have perfect order classes,
it follows that the subgroup $T$ must be non-trivial.
Therefore
\begin{displaymath}
\order{H} = 2^n\order{T} > 2^n.
\end{displaymath}
Every element of order $4$ in $H$ belongs to $Q$,
so we have
\begin{displaymath}
\ordcount{4}{H} = \ordcount{4}{Q} = 2^{n-1} + 2 = 2(2^{n-2} + 1).
\end{displaymath}
Since $n\geq 3$, hence, $2^{n-2} + 1$ is an odd number greater than $1$.
And, since $\ordcount{4}{H}$ divides the order of $H$,
it follows that $2^{n-2} + 1$ is a divisor of the order of $T$.

We claim that $2^{n-2} + 1$ is square-free.
To see this, suppose that there is a prime $q$ such that $q^2$
divides $2^{n-2} + 1$,
so that $q^2$ is also a divisor of $\order{T}$.
Let $q^m$ be the largest power of $q$ dividing the order of $T$,
where $m\geq 2$ is an integer,
Since $T$ is a cyclic group, therefore, $T$ has
$\eulerphi(q^m) = q^{m-1}(q - 1)$ elements of order $q^m$,
and we have
\begin{displaymath}
\ordcount{4q^m}{H} = \ordcount{4}{Q}\ordcount{q^m}{T} = 2(2^{n-2} + 1)q^{m-1}(q - 1).
\end{displaymath}
Since $q^2$ divides $2^{n-2} + 1$, it follows that $q^{m+1}$
divides $\ordcount{4q^m}{H}$,
which implies that $q^{m+1}$ divides $\order{T}$,
in contradiction to the choice of $m$.
Consequently, $2^{n-2} + 1$ must be square-free.

Next, let $p$, $q$ and $r$ be different prime divisors of the order of $T$.
Then the order of $H$ is divisible by
\begin{displaymath}
\ordcount{2^{n-1}pqr}{H} = \ordcount{2^{n-1}}{Q_n}\ordcount{pqr}{T} = 2^{n-2}(p-1)(q-1)(r-1).
\end{displaymath}
Since $p$, $q$ and $r$ are odd, we see that $\ordcount{2^{n-1}pqr}{H}$
is divisible by $2^{n+1}$, which cannot divide $\order{H} = 2^n\order{T}$.
Now, if $2^{n-2} + 1$ had three distinct prime divisors,
then so too would $\order{T}$,
so it must be that $2^{n-2} + 1$ is a prime or a product of two distinct primes.

Suppose that
\begin{displaymath}
2^{n-2} + 1 = pq,
\end{displaymath}
where $p$ and $q$ are (odd) primes with $p < q$.
Let $a$ and $b$ be positive integers such that
$\order{T} = p^{a}q^{b}$, so that
\begin{displaymath}
\order{H} = 2^{n}p^{a}q^{b}.
\end{displaymath}
Then
\begin{displaymath}
\ordcount{2^{n-1}pq}{H} = \ordcount{2^{n-1}}{Q}\ordcount{pq}{T} = 2^{n-2}(p - 1)(q - 1),
\end{displaymath}
and so we must have
\begin{displaymath}
2^{n-2}(p - 1)(q - 1) \mid  2^{n}p^{a}q^{b}.
\end{displaymath}
From this we see that the even numbers $p - 1$ and $q - 1$ must be
exactly divisible by $2$ (that is, that $4$ divides neither).
Because $p - 1 < p < q$, it follows that $p - 1$ can only be equal to $2$, so $p = 3$.
We therefore have
\begin{displaymath}
2^{n-1}(q - 1) \mid 2^{n}3^{a}q^{b},
\end{displaymath}
so $q - 1 = 2\cdot 3^{i}$, for some non-negative integer $i$,
and $q = 1 + 2\cdot 3^i$ is a Pierpont prime.
Now,
\begin{displaymath}
2^{n-2} + 1 = 3q = 3(1 + 2\cdot 3^i) = 3 + 2\cdot 3^{i+1};
\end{displaymath}
whence, $2\cdot 3^{i+1} = 2^{n-2} - 2 = 2(2^{n-3} - 1)$,
so $3^{i+1} = 2^{n-3} - 1$, or
\begin{displaymath}
2^{n-3} - 3^{i+1} = 1.
\end{displaymath}
By Lemma~\ref{lem:cons2a3b}, we have
\begin{displaymath}
(n,i) \in \{ (4, -1), (5, 0) \}.
\end{displaymath}
Since $i = -1$ is absurd, it must be that $i = 0$.
But then we are led to the further absurdity $q = 3 > p = 3$.
We are forced to conclude, therefore, that $2^{n-2} + 1$ is a prime.

Let $p = 2^{n-2} + 1$.
Now, $T$ has $\eulerphi(p) = p - 1 = 2^{n-2}$ elements of order $p$,
and
\begin{displaymath}
\ordcount{2^{n-1}p}{H} = \ordcount{2^{n-1}}{Q}\ordcount{p}{T} = 2^{n-2}\cdot 2^{n-2} = 2^{2n - 4}.
\end{displaymath}
Since this quantity must divide $\order{H} = 2^n\order{T}$,
therefore, $2n - 4\leq n$, so $n\leq 4$.
But also, $n$ is greater than or equal to $3$ so,
in fact, $n = 3$ or $n = 4$.
If $n = 3$, then $p = 2^{n-2} + 1 = 3$; and,
if $n = 4$, then $p = 2^{n-2} + 1 = 5$.

We know that the order of $T$ cannot be divisible by
three primes, but we have yet to rule out the possibility
that $\order{T}$ has a prime divisor other than $p$.
In the case that $p = 3$ and $n = 3$, this follows from
\cite{McCarron2020} since, in this case, $H$ is a Hamiltonian
group with perfect order classes, so $T$ must be a non-trivial
cyclic $3$-group.
So suppose that $p = 5$ and $n = 4$,
and also that $T$ has a prime divisor $q$ other than $5$.
Then
\begin{displaymath}
\ordcount{40q}{H} = \ordcount{8}{Q}\ordcount{5q}{T} = 4\cdot 4\cdot (q-1),
\end{displaymath}
which, since $q$ is odd, is divisible by $32$.
It follows that $\ordcount{40q}{H}$ cannot divide $\order{H} = 16\order{T}$
in this case, so $T$ must be a cyclic $5$-group.

Thus, in summary, $H$ is isomorphic either to
$Q_{3}\times\cyclic{3^k}$ or to $Q_{4}\times\cyclic{5^k}$,
for some positive integer $k$.

Conversely, we saw in Example~\ref{ex:q4c5k} that
$Q_{4}\times\cyclic{5^k}$ has perfect order classes.
That $Q_{3}\times\cyclic{3^k}$ has perfect order classes
follows from~\cite{McCarron2020}, or by the method of
Example~\ref{ex:q4c5k}.
\end{proof}

Now that we have a classification of the non-abelian nilpotent
complements with perfect order classes, we shall see next that
none of them ``fits'' into a Frobenius group that has perfect order classes.

\begin{theorem}\label{thm:nonilpcompl}
A nilpotent Frobenius complement in a Frobenius group with perfect order classes
is cyclic.
\end{theorem}
\begin{proof}
Let $G = K\sdpf H$ be a Frobenius group with perfect order classes,
and assume, for an eventual contradiction,
that the Frobenius complement $H$ is a non-abelian nilpotent group.
By Theorem~\ref{thm:main}, $H$ has perfect order classes,
$K$ is a homocyclic group of the form $\cyclic{p^k}^r$,
for an odd prime $p$ and positive integers $k$ and $r$,
and we have $\order{H} = p^r - 1$.
By Theorem~\ref{thm:nilpcompl}, $H$ is isomorphic to
$Q_{n}\times\cyclic{q^t}$, for some positive integer $t$,
where $(n,q) \in \{ (3,3), (4,5) \}$.
Thus,
\begin{displaymath}
p^r - 1 = 2^{n}q^{t}.
\end{displaymath}
Note that $r$ must be greater than $1$ since $H$ is not abelian.

Suppose first that $r$ is odd, so that $r\geq 3$.
If $q = 3$, then $p > 3$, and so $3$ divides $p^2 - 1$,
and $2$ is of course also a divisor of $p^2 - 1$.
By Zsigmondy's theorem, $p^r - 1$ has a primitive
prime divisor other than $2$ and $3$,
in contradiction to the equality $p^r - 1 = 2^{4}\cdot 3^{t}$.
It must be then that $q = 5$ and $n = 4$,
in consequence of which we have either that $p = 3$ or that $p > 5$.
If $p = 3$, then
\begin{displaymath}
2^{4}\cdot 5^{t} = 3^r - 1 = 2(1 + 3 + 3^2 + \cdots + 3^{r-1}).
\end{displaymath}
Since $r$ is odd, therefore, $1 + 3 + \cdots + 3^{r-1}$ is odd,
which implies the absurdity $2 = 16$.
Therefore, it must be that $p$ is greater than $5$.

Now if $r = 3$, then
\begin{displaymath}
2^{4}\cdot 5^{t} = p^3 - 1 = (p - 1)(p^2 + p + 1)
\end{displaymath}
and, since $p^2 + p + 1$ is odd, it follows that $16$ divides $p - 1$,
and $p^2 + p + 1 = 5^i$, for some integer $i$ satisfying $0\leq i\leq t$.
It is clear that $i$ must be positive, however, so we have
\begin{displaymath}
p(p + 1) = 5^i - 1 = 4(1 + 5 + \cdots + 5^{i-1}).
\end{displaymath}
Thus, $4$ divides $p + 1$.
But this means that $4$ is a divisor of both $p - 1$ and $p + 1$,
which is impossible.
Thus, $r$ cannot be equal to $3$ and we must have $r \geq 5$.
But $5$ divides $p^4 - 1$, by Euler's theorem
and $2$ certainly is a divisor of $p - 1$,
so Zsigmondy's theorem implies that $p^r - 1$ has a prime divisor
other than $2$ and $5$, contradicting $p^r - 1 = 2^{4}\cdot 5^t$.
We conclude, therefore, that $r$ must be even.

Write $r = 2s$, where $s$ is a positive integer.
Then
\begin{displaymath}
2^n q^t = p^{2s} - 1 = (p^s - 1)(p^s + 1),
\end{displaymath}
from which it follows that both $p^s - 1$ and $p^s + 1$
are $\{2,q\}$-numbers.

If $q = 3$ and $n = 3$ then, by Corollary~\ref{cor:cons23nums},
we have $p^s \in \{2,3,5,7,17 \}$,
whence, $s = 1$ and $r = 2$, and in fact $p\in \{ 5, 7, 17 \}$,
since $p$ is odd and $q = 3$ forces $p > 3$.
If $p = 7$, then $p^2 - 1 = 48 = 16\cdot 5$,
which cannot be equal to $8\cdot 3^t$, for any $t$.
If $p = 17$, then $p^2 - 1 = 288 = 2^{5}\cdot 3^{2}\neq 8\cdot 3^t$,
regardless of the value of $t$.
It must therefore be that $p = 5$, so that
\begin{displaymath}
8\cdot 3^t = 5^2 - 1 = 24,
\end{displaymath}
and hence, $t = 1$.
Then $H \iso Q_{3}\times\cyclic{3}$ and $K\iso\cyclic{5^k}\times\cyclic{5^k}$,
for some positive integer $k$.
Presuming the existence of such a group,
we infer the existence of its Frobenius subgroup
$\Omega_{1}(K)H$ of order $25\cdot 24 = 600$ and with
Frobenius complement $H\iso Q_{3}\times\cyclic{3}$.
(Note that we do not assert that $\Omega_{1}(K)H$ has perfect order classes.)
However, consulting the library of Frobenius groups in \Maple{},
we see that none of the Frobenius groups of order $600$ has a Frobenius
complement isomorphic to $Q_{3}\times\cyclic{3}$.
(They are isomorphic to $\cyclic{3}\sdp\cyclic{8}$, $\cyclic{24}$ or $\SL{2}{3}$.)
We conclude, therefore, that no such group $G$ exists.

Now suppose that $q = 5$ and $n = 4$,
so $H\iso Q_{4}\times\cyclic{5^t}$,
and we have $p^r - 1 = 16\cdot 5^t$.
Suppose that $p > 3$, so that $3$ divides $p^2 - 1$.
Since $r$ is even, hence, $p^2 - 1$ divides $p^r - 1$,
so $3$ also divides $p^r - 1$.
Therefore, $p^r - 1$ cannot be a $\{2,5\}$-number,
so we must have $p = 3$, and hence,
\begin{displaymath}
2^{4}\cdot 5^t = 3^r - 1.
\end{displaymath}
Since $3^r\equiv 1\pmod{5}$,
it follows that $r$ is a multiple of $4$,
so we can write $r = 4s$,
for a positive integer $s$.
Suppose that $s > 1$, so that $r > 4$.
Since $2$ and $5$ both divide $3^4 - 1 = 80$,
Zsigmondy's theorem implies that $3^r - 1$
has a prime divisor other than $2$ and $5$,
in contradiction to $3^r - 1 = 2^{4}\cdot 5^{t}$.
Thus, $r$ must be equal to $4$, and we have
\begin{displaymath}
2^{4}\cdot 5^t = 3^4 - 1 = 80 = 2^{4}\cdot 5,
\end{displaymath}
whence $t = 1$ and $H\iso Q_{4}\times\cyclic{5}$,
and also $K\iso\cyclic{3^k}^4$,
for some positive integer $k$.
Now $\Omega_{1}(K)H$ is a Frobenius subgroup of $G$
of order $6480$, since $\Omega_{1}(K)\iso\cyclic{3}^4$.
However, from the library of Frobenius groups in \Maple{},
we see that no Frobenius group of order $6480$ has
Frobenius complement isomorphic to $Q_{4}\times\cyclic{5}$.
It follows, therefore, that $G$ does not exist.

Having thus exhausted every possibility,
the proof is complete.
\end{proof}

\begin{proof}[Proof of Theorem C]
It follows from Theorem~\ref{thm:nonilpcompl} that the Frobenius
groups with a nilpotent Frobenius complement and perfect order classes
are exactly those described in Proposition~\ref{prop:cycliccompl}.
\end{proof}

We now have a complete picture of the Frobenius groups with
perfect order classes whose Frobenius complement is nilpotent.
In the following sections, we begin to study soluble Frobenius
groups with perfect order classes whose Frobenius complement is
not nilpotent.

\subsection{Frobenius Groups with Non-Nilpotent Biprimary Complement}\label{ss:2qcompl} 

For Frobenius groups whose Frobenius complement is not nilpotent,
we begin by considering cases for which the order of the Frobenius
complement has just one odd prime divisor.
By Burnside's theorem, any such group is soluble.

We begin by characterising Frobenius groups with perfect order classes,
assuming that the Frobenius complement is a non-nilpotent $\{2, 3\}$-group
or a non-nilpotent $\{2, 5\}$-group.

\begin{theorem}\label{thm:sol23compl}
Let $G = K\sdpf H$ be a Frobenius group
and assume that the Frobenius complement $H$ is a non-nilpotent $\{2,3\}$-group.
Then $G$ has perfect order classes if, and only if,
$G$ has one of the following forms:
\begin{enumerate}
\item{$G\iso\cyclic{5^k}^2\sdpf\SL{2}{3}$;}
\item{$G\iso\cyclic{5^k}^2\sdpf\langle a,b \mid a^3, b^8, a^b = a^{-1}\rangle$;}
\item{$G\iso\cyclic{7^k}^2\sdpf\langle a,b \mid a^3, b^{16}, a^b = a^{-1}\rangle$; or}
\item{$G\iso\cyclic{17^k}^2\sdpf\langle a,b \mid a^9, b^{32}, a^b = a^{-1}\rangle$,}
\end{enumerate}
for some positive integer $k$.
\end{theorem}

\begin{proof}
Suppose that $G$ has perfect order classes.
Then $H$ also has perfect order classes,
$K\iso\cyclic{p^k}^r$ is a homocyclic $p$-group for some odd prime $p$,
and positive integers $k$ and $r$,
and $\order{H} = p^r - 1$.
Since $H$ is non-abelian, it follows that the rank $r$ of $K$ is greater than $1$.
It must also be that the order of $H$ is divisible by $3$ since,
otherwise, $H$ would be a $2$-group, and we have assumed that
$H$ is not nilpotent.

First assume that $r$ is odd, so that $r\geq 3$.
Since $p > 3$, it follows that $p^2 - 1$ is divisible by $3$ and,
of course, $p-1$ is divisible by $2$.
By Zsigmondy's theorem, $p^r - 1$ has a prime divisor $q$
such that $q$ does not divide $p^i - 1$, for $1\leq i < r$.
In particular, $q\neq 2$ and $q\neq 3$.
But this contradicts $p^r - 1 = \order{H}$ being a $\{2,3\}$-number.
Therefore, $r$ must be even and we can write $r = 2s$,
for some positive integer $s$.
Then we have
\begin{displaymath}
\order{H} = p^{2s} - 1 = (p^s - 1)(p^s + 1);
\end{displaymath}
whence, both $p^s - 1$ and $p^s + 1$ are $\{2,3\}$-numbers.
Therefore, by Corollary~\ref{cor:cons23nums}, we have
$p^s\in\{2,3,5,7,17\}$, whence $s = 1$ (and $r = 2$),
and since $p$ is greater than $3$ we have $p\in \{5, 7, 17\}$.
If $p = 5$, then $\order{H} = p^2 - 1 = 24$;
if $p = 7$, then $\order{H} = 48$; and,
if $p = 17$, then $\order{H} = 288$.

Consulting the database of small groups in \Maple{}, we find that
the non-nilpotent groups of order $24$ with perfect order classes are
$\SL{2}{3}$,
$\langle a,b\mid a^3, b^8, a^b = a^{-1}\rangle\iso\cyclic{3}\sdp\cyclic{8}$,
and
$\langle a,b,c\mid a^4,b^2,c^3, c^a = c^2, [a,b], [b,c] \rangle \iso \cyclic{2}\times(\cyclic{3}\sdp\cyclic{4})$.
The last of these three contains a subgroup of the form $\cyclic{2}\times\cyclic{2}$,
so it cannot be a Frobenius complement.
The first two do occur as Frobenius complements.

The non-nilpotent groups of order $48$ with perfect order classes are
\begin{displaymath}
\langle a,b \mid a^{3}, b^{16}, a^b = a^{-1} \rangle \iso \cyclic{3}\sdp\cyclic{16},
\end{displaymath}
\begin{displaymath}
\langle a,b,c \mid a^8, b^2, c^3, [a,b], c^a = c^2, [b,c] \rangle \iso (\cyclic{3}\sdp\cyclic{8})\times\cyclic{2},
\end{displaymath}
\begin{displaymath}
\langle a,b,c \mid a^8, b^2, c^3, c^a = c^2, [a,b], [b,c], aba^3b, a^{-1}ba^2ba^{-1}\rangle \iso (\cyclic{3}\sdp\cyclic{8})\sdp\cyclic{2},
\end{displaymath}
as well as
$\cyclic{2}\times\SL{2}{3}$.
The last three of these contain a subgroup of the form $\cyclic{2}\times\cyclic{2}$,
so cannot be Frobenius complements.
Therefore, the only choice for a group of order $48$ among them
is the first, with structure $\cyclic{3}\sdp\cyclic{16}$.

Similarly, among the $26$ non-nilpotent groups of order $288$ with
perfect order classes only one group,
\begin{displaymath}
\langle a, b \mid a^9, b^{32}, a^b = a^{-1} \rangle \iso \cyclic{9}\sdp\cyclic{32}
\end{displaymath}
can be a Frobenius complement.
For instance, there are two groups
\begin{displaymath}
G_{188} = \langle a,b,c \mid a^3, b^3, [a,b], c^{32}, a^c = ab^{-1}, b^c = (ba)^{-1} \rangle \iso (\cyclic{3}\times\cyclic{3})\sdp\cyclic{32}
\end{displaymath}
and
\begin{displaymath}
G_{373} = \langle a,b,c \mid a^3, b^3, [a,b], c^{32}, [a,c], b^c = ab^{-1} \rangle \iso (\cyclic{3}\times\cyclic{3})\sdp\cyclic{32}
\end{displaymath}
which cannot be Frobenius complements because they have
subgroups of the form $\cyclic{3}\times\cyclic{3}$.
The remaining $23$ groups all have a subgroup of the form $\cyclic{2}\times\cyclic{2}$.

For the converse, assume that $H$ has perfect order classes,
and that $K\iso\cyclic{p^k}^2$,
where $(p,\order{H}) \in \{ (5, 24), (7, 48), (17, 288) \}$.
(Of course, $H$ must be among those groups identified above.)
Since $H$ has perfect order classes, we need only check
that the number of elements whose order is a divisor of
$\order{K} = p^{2k}$ divides the order of $G$.
For $1\leq i \leq k$, we have
\begin{displaymath}
\ordcount{p^i}{G} = \ordcount{p^i}{K} = p^{2(i-1)}(p - 1).
\end{displaymath}
If $p = 5$, then $p - 1 = 4$, which divides $24$.
If $p = 7$, then $p - 1 = 6$, which divides $48$.
If $p = 17$, then $p - 1 = 16$, which divides $288$.
Finally, $\order{H} = p^2 - 1$ in each case.
Therefore, $G$ has perfect order classes.
\end{proof}

\begin{theorem}\label{thm:sol25compl}
Let $G$ be a Frobenius group whose Frobenius complement $H$
is a non-nilpotent $\{2,5\}$-group.
Then $G$ has perfect order classes if, and only if,
the Frobenius kernel is a homocyclic $3$-group of rank $4$,
and $H\iso\langle a,b\mid a^5, b^{16}, a^b = a^4\rangle$.
\end{theorem}
\begin{proof}
One checks directly that the group
$\langle a,b \mid a^5, b^{16}, a^b = a^{-1} \rangle$
is a non-nilpotent group of order $80 = 3^{4} - 1$
with perfect order classes.
Therefore, if $H$ is isomorphic to this group,
and if the Frobenius kernel of $G$ is a homocyclic $3$-group of rank $4$,
then $G$ has perfect order classes by Theorem~\ref{thm:main}.

To prove the converse, assume that $G$ has perfect order classes.
Then by Theorem~\ref{thm:main}, the Frobenius kernel $K$ of $G$ is a homocyclic $p$-group
of the form $\cyclic{p^k}^r$, for some odd prime $p$
and suitable positive integers $k$ and $r$,
$H$ has perfect order classes,
and we have $\order{H} = p^r - 1$.
Since $H$ is assumed not nilpotent,
both $2$ and $5$ divide the order of $H$,
and the rank $r$ of $K$ must be greater than $1$.
Since the orders of $H$ and $K$ are relatively prime,
therefore, $p$ cannot be equal to $5$.

Suppose first that $r$ is odd, so in fact,
$r$ must be greater than or equal to $3$.
If $r = 3$, then $p^3 - 1 = \order{H}$ is a $\{2,5\}$-number divisible by $10$,
so $p^3 \equiv 1\pmod{5}$.
Together with $p^4\equiv 1\pmod{5}$, from Euler's theorem,
this implies that $p\equiv 1\pmod{5}$,
that is, that $5$ divides $p - 1$.
But $2$ also divides $p - 1$,
so Zsigmondy's theorem implies that there is a prime
divisor $q$ of $p^3 - 1$ that is not a divisor of $p - 1$
(or of $p^2 - 1$).
In particular, $q$ cannot be either of $2$ and $5$,
in contradiction to the assumption that
$p^3 - 1 = \order{H}$ is a $\{2,5\}$-number.
Therefore, $r$ cannot be equal to $3$, and we must have $r\geq 5$.
But, from the divisibility of $p^4 - 1$ by $5$,
another application of Zsigmondy's theorem shows that
$r$ cannot be odd.

So, $r$ is even.
If $p$ is not equal to $3$,
then $p$ must be greater than $5$ and $3$
divides $p^2 - 1$ which, in turn, divides $p^r - 1$.
Since $\order{H} = p^r - 1$, this contradicts the
assumption that $H$ is a $\{2,5\}$-group.
Consequently, it must be that $p$ is equal to $3$.

From $3^r \equiv 1\pmod{5}$ it follows that $r$ is a multiple of $4$,
and so we can write $r = 4s$, for some positive integer $s$.
If $s = 1$ so that $r = 4$, then we have
\begin{displaymath}
\order{H} = 3^4 - 1 = 80 = 2^{4}\cdot 5.
\end{displaymath}
If $s > 1$, so that $r \geq 8$,
then we again apply Zsigmondy's theorem to obtain a prime
divisor $q$ of $p^r - 1$ such that $q$ does not divide
$p^i - 1$, for $1\leq i < r$.
In particular, $q$ does not divide $3^4 - 1 = 80$,
so $q$ cannot be either of $2$ and $5$ and this
contradicts that assumed fact that $3^r - 1$ must
be a $\{2,5\}$-number.
Therefore, the only possibility is that the rank $r$
of $K$ is equal to $4$,
giving $K\iso\cyclic{3^k}^4$,
for a suitable positive integer $k$,
and also that $H$ is a group of order $80$.

Now, consulting the library of small groups in \Maple{} we find that,
among the groups of order $80$, just three have perfect order classes,
and one of those, $Q_{4}\times\cyclic{5}$, is nilpotent,
leaving the two metacyclic groups
\begin{displaymath}
A = \langle a, b \mid a^5, b^{16}, a^b = a^2\rangle
\textrm{ and }
B = \langle a, b \mid a^5, b^{16}, a^b = a^4\rangle.
\end{displaymath}
Both $A$ and $B$ are indeed Frobenius complements,
but only $B$ occurs as a Frobenius complement
in a Frobenius group with Frobenius kernel
isomorphic to $\cyclic{3^k}^4$.
To see this, suppose that there is a Frobenius
group $F \iso \cyclic{3^k}^4\sdpf A$.
Then $F$ has a Frobenius subgroup isomorphic
to $\cyclic{3}^4\sdpf A$ of order $6480$.
Checking the \Maple{} library of Frobenius groups,
we find that there are three Frobenius groups of
order $6480$ of which two have a Frobenius complement
of order $80$.
One of those complements is the cyclic group of that order,
while the other is our group $B$.
Therefore, we must have $H\iso B$.
\end{proof}

\begin{remark}
The proof of Theorem~\ref{thm:sol25compl} furnishes us with an example
\begin{displaymath}
A = \langle a,b\mid a^5, b^{16}, a^b = a^2\rangle,
\end{displaymath}
of a soluble but non-nilpotent Frobenius complement with perfect order classes
that is not a Frobenius complement in a Frobenius group with
perfect order classes.
\end{remark}

Having dealt with Frobenius complements that are either
$\{2,3\}$-groups or $\{2,5\}$-groups,
let us now show that there are no other $\{2,q\}$-groups
to be considered.

\begin{lemma}\label{lem:cons2qnums}
Let $q$ be an odd prime, and let $x$, $y$, $u$ and $v$ be non-negative
integers such that
\begin{displaymath}
2^xq^y - 2^uq^v = 1.
\end{displaymath}
\begin{enumerate}
\item{If $q = 2^p - 1$ is a Mersenne prime,
where $p$ is a prime, then
\begin{displaymath}
(x,y,u,v) \in \{ (1,0,0,0), (p,0,0,1) \}.
\end{displaymath}
}
\item{If $q = 2^{2^t} + 1$ is a Fermat prime,
for a non-negative integer $t$, then
\begin{displaymath}
(x,y,u,v) \in \{ (1,0,0,0), (0,1,1,0), (0,1,2^t,0) \}.
\end{displaymath}
}
\item{Otherwise,
\begin{displaymath}
(x,y,u,v) \in \{ (1,0,0,0), (0,1,1,0), (0,2,3,0) \}.
\end{displaymath}
}
\end{enumerate}
In all cases, if $(x,y,u,v) \in \{ (0,1,1,0), (0,2,3,0) \}$, then $q = 3$.
\end{lemma}
\begin{proof}
Since $2^xq^y$ and $2^uq^v$ are consecutive integers,
one must be even and the other odd,
so exactly one of $x$ and $u$ is equal to zero.

Suppose first that $x = 0$ and $u > 0$,
so we have
\begin{displaymath}
1 = q^y - 2^uq^v.
\end{displaymath}
If $y\leq v$, then $1 = q^y(1 - 2^uq^{v-y})$,
so $y = 0$ and $1 = 1 - 2^uq^v$.
But this implies the impossibility $2^uq^v = 0$,
so we must have $v < y$.
Then
\begin{displaymath}
1 = q^v(q^{y-v} - 2^u);
\end{displaymath}
whence, $v = 0$ and
\begin{displaymath}
q^y = 2^u + 1.
\end{displaymath}
If $u = 0$, then $q^y = 2$, contradicting the assumption that $q$ is odd.
If $u = 1$, then $q^y = 3$, so $q = 3$ and $y = 1$.
In this case, we have
\begin{displaymath}
(x,y,u,v) = (0,1,1,0), \;\; q = 3.
\end{displaymath}
Assume now that $u > 1$.
If $y = 0$, we arrive at the impossible $2^u + 1 = 1$.
If $y = 1$, then $q = 2^u + 1$ is a Fermat prime,
so there is a non-negative integer $t$ for which $u = 2^t$.
Then
\begin{displaymath}
(x,y,u,v) = (0,1,2^t,0), \;\; q = 2^{2^t} + 1.
\end{displaymath}
Otherwise, $y > 1$ and, since $u > 1$ also,
the only solution is $q = 3$, $u = 3$ and $y = 2$,
so we have
\begin{displaymath}
(x,y,u,v) = (0,2,3,0), \;\; q = 3.
\end{displaymath}

Now suppose that $u = 0$ and $x > 0$,
so we have
\begin{displaymath}
1 = 2^xq^y - q^v.
\end{displaymath}
If $v\leq y$, then
\begin{displaymath}
1 = q^v(2^xq^{y-v} - 1);
\end{displaymath}
whence, $v = 0$ and $1 = 2^xq^y - 1$,
or $2^xq^y = 2$.
Then $x = 1$ and $y = 0$ and we have
\begin{displaymath}
(x,y,u,v) = (1,0,0,0).
\end{displaymath}
If $y < v$, then
\begin{displaymath}
1 = q^y(2^x - q^{v-y});
\end{displaymath}
whence, $y = 0$ and
\begin{displaymath}
q^v = 2^x - 1.
\end{displaymath}
It is clear that $x$ cannot be zero and,
if $x = 1$, then $q^v = 1$, so $v = 0$,
Thus,
\begin{displaymath}
(x,y,u,v) = (1,0,0,0)
\end{displaymath}
in this case.
Otherwise, we have $x > 1$,
in which case it is clear that $v > 0$.
If $v = 1$, then $q = 2^x - 1$ is a Mersenne prime,
so $x$ is a prime $p$ and we have
\begin{displaymath}
(x,y,u,v) = (p, 0, 0, 1), \;\; q = 2^p - 1.
\end{displaymath}
Otherwise, we have $v > 1$ and, since $x > 1$,
there is no solution in this case.
\end{proof}

\begin{proposition}\label{prop:cons2qnums}
Let $q$ be an odd prime, and let $n$ be a positive integer
such that $n^2 - 1$ is divisible by $2$ and $q$, but by
no other primes.
Then one of the following is true.
\begin{enumerate}
\item{If $q = 2^p - 1$ is a Mersenne prime, where $p$ is a prime,
then $n = 2^{p+1} - 1$; and, if $n$ is a prime power, then $(q,n) = (3,7)$.}
\item{If $q = 2^{2^t} + 1$ is a Fermat prime, then $n = 2^{1 + 2^t} + 1$;
and, if $n$ is a prime power, then $(q,n)\in\{ (3,5), (5,9) \}$.}
\item{Otherwise, $(q,n) \in \{ (3,5), (3,17) \}$.}
\end{enumerate}
In particular, if $n$ is a prime power, then $q\in \{3,5\}$.
\end{proposition}
\begin{proof}
Note that $n > 1$ since $0$ is divisible by every prime,
and $n$ must be odd, since $n^2 - 1$ is even.
Since $n^2 - 1 = (n+1)(n-1)$ and both $n + 1$ and $n - 1$ are
even positive integers, we can write
\begin{displaymath}
n + 1 = 2^aq^b \; \mathrm{and} \; n - 1 = 2^cq^d,
\end{displaymath}
where $a$ and $c$ are positive integers,
and $b$ and $d$ are non-negative integers.
(In fact, one of $a$ and $c$ must be equal to $1$,
and one of $b$ and $d$ must be zero.)
Since $a$ and $c$ are positive, we have
\begin{displaymath}
2 = (n + 1) - (n - 1) = 2^aq^b - 2^cq^d,
\end{displaymath}
and hence,
\begin{displaymath}
2^{a-1}q^b - 2^{c-1}q^d = 1,
\end{displaymath}
with all of $a-1$, $b$, $c-1$ and $d$ non-negative.
We now consider the cases from Lemma~\ref{lem:cons2qnums}.

If $(a-1, b, c-1, d) = (1,0,0,0)$, then $n = 3$.
But $3^2 - 1 = 8$ is not divisible by any odd prime,
so this case does not occur.
If $(a-1,b,c-1,d) = (0,1,1,0)$, then $q = 3$ and $n = 5$.
If $(a-1,b,c-1,d) = (0,2,3,0)$, then $n = 2q^2 - 1 = 17$, from which we obtain $q = 3$.

If $q = 2^p - 1$ is a Mersenne prime, and $(a-1,b,c-1,d) = (p,0,0,1)$,
then $n = 2^{p+1} - 1$.
If $n$ is prime, then $p + 1$ is prime, so $p = 2$,
$q = 3$ and $n = 7$.
In this case, $n$ cannot be a proper power.

If $q = 2^{2^t} + 1$ is a Fermat prime and $(a-1,b,c-1,d) = (0,1,2^t,0)$,
then $n = 2^{1+2^t} + 1$.
If $n$ is a prime, it is a Fermat prime,
so $1 + 2^t$ is a power of $2$, say, $1 + 2^t = 2^i$.
Then $t = 0$ and $i = 1$, so $q = 3$ and $n = 5$.
If $n$ is a proper prime power, then $n - 1 = 2^{1+2^t}$,
so $n = 9$ and $2^{1+2^t} = 8$,
which implies that $3 = 1 + 2^t$, so $t = 1$ and $q = 5$.
\end{proof}

\begin{theorem}\label{thm:nn2qis35}
Let $G = K\sdpf H$ be a Frobenius group with perfect order classes,
and suppose that $H$ is a non-nilpotent $\{2,q\}$-group,
where $q$ is an odd prime.
Then $q = 3$ or $q = 5$ and the rank of $K$ is either $2$ or $4$.
\end{theorem}
\begin{proof}
Since $H$ is a non-nilpotent $\{2,q\}$-group,
we can write $\order{H} = 2^aq^b$,
where $a$ and $b$ are positive integers.
By Theorem~\ref{thm:main}, $H$ has perfect order classes,
$K\iso\cyclic{p^k}^r$, for some odd prime $p$ and positive
integers $k$ and $r$, and we have
\begin{displaymath}
p^r - 1 = \order{H} = 2^aq^b.
\end{displaymath}

Since $H$ has perfect order classes, it follows that $q - 1$
divides $\order{H} = 2^aq^b$ and,
since $\gcd(q,q-1) = 1$, hence, $q - 1$ is a divisor of $2^a$,
which means that it is a power of $2$,
say $q - 1 = 2^{\tau}$, where $\tau$ is a positive integer.
Then $q = 2^{\tau} + 1$ is a Fermat prime,
so $\tau = 2^t$ is a power of $2$, for some non-negative integer $t$.

By Euler's theorem, we have $p^{q-1}\equiv 1\pmod{q}$,
so $q$ divides $p^{q-1} - 1$.
Since $q$ also divides $p^r - 1$, it follows that $q$ is a divisor of
\begin{displaymath}
\gcd( p^r - 1, p^{q-1} - 1) = p^d - 1,
\end{displaymath}
where $d = \gcd( r, q-1 )$.
If $d < r$ then, since $2$ divides $p - 1$ and $q$ divides $p^d - 1$,
Zsigmondy's theorem implies that $p^r - 1$ has a prime divisor
other than $2$ and $q$, contrary to $p^r - 1 = 2^aq^b$.
Therefore, $r = d = \gcd(r, q-1)$,
so $r$ divides
\begin{displaymath}
q - 1 = (2^{2^t} + 1) - 1 = 2^{2^t}.
\end{displaymath}
Thus, $r$ is a power of $2$.
Since $H$ is not abelian, hence, $r > 1$ and,
in particular, $r$ is even.
Applying Proposition~\ref{prop:cons2qnums},
we see that $q \in \{3, 5\}$ and,
since $r > 1$ divides $q - 1$, it follows that $r = 2$
or $r = 4$.
\end{proof}

\subsection{Frobenius Groups with Soluble $\{2,3,5\}$-Complement That Do Not Exist}\label{ss:sol235compl} 

We saw in Theorem~\ref{thm:perfect} that a Frobenius complement of
an insoluble Frobenius group with perfect
order classes is isomorphic to the $\{2,3,5\}$-group $\SL{2}{5}$.
This leads us to ask whether there are soluble Frobenius groups with
perfect order classes whose Frobenius complement is a $\{2,3,5\}$-group.
We conclude by showing that there are no such soluble examples.

We begin, as usual, with some number theoretic preliminaries.

\begin{lemma}\label{lem:23minus25}
Suppose that $x$, $y$, $u$ and $v$ are non-negative integers such that
\begin{displaymath}
2^x\cdot 3^y - 2^u\cdot 5^v = 1.
\end{displaymath}
Then
\begin{displaymath}
(x,y,u,v) \in \{ (0,4,4,1), (1,1,0,1), (1,0,0,0), (0,1,1,0), (0,2,3,0) \}.
\end{displaymath}
\end{lemma}
\begin{proof}
First assume that $x\leq u$.
Then
\begin{displaymath}
1 = 2^x\cdot 3^y - 2^u\cdot 5^v = 2^x(3^y - 2^{u-x}\cdot 5^v),
\end{displaymath}
so we must have $2^x = 1 = 3^y - 2^{u-x}\cdot 5^v$,
whence $x = 0$ and
\begin{displaymath}
2^u\cdot 5^v = 3^y - 1.
\end{displaymath}
If $v = 0$ then this becomes $2^u = 3^y - 1$,
from which we conclude, with the help of Lemma~\ref{lem:cons3b2a},
either that $y = u = 1$ or that $y = 2$ and $u = 3$.
In this way we obtain the solutions
\begin{displaymath}
(x,y,u,v) \in \{ (0,1,1,0), (0,2,3,0) \}.
\end{displaymath}
Otherwise, $v$ is positive, so $3^y\equiv 1\pmod{5}$,
from which it follows that $y$ is a multiple of $4$.
Write $y = 4t$, where $t$ is a positive integer.
Then
\begin{displaymath}
2^u\cdot 5^v = 3^y - 1 = (3^{2t})^2 - 1  = (3^{2t} - 1)(3^{2t} + 1).
\end{displaymath}
Since $\gcd( 3^{2t} - 1, 3^{2t} + 1) = 2$,
hence, $3^{2t} - 1$ and $3^{2t} + 1$ have no common odd
prime divisors.
Thus, one of $3^{2t} - 1$ and $3^{2t} + 1$ must be a power
of $2$, while the other is exactly divisible by $2$.
Now, $3^{2t} + 1 = 2^{u-1}$ implies that $2t = 1$,
which is not possible.
Hence, $3^{2t} - 1 = 2^{u-1}$, from which we get $2t = 2$
and $u-1 = 3$.
Therefore, $t = 1$ and $y = 4$, while $u = 4$.
This gives, from $3^y - 2^u\cdot 5^v = 1$ that
\begin{displaymath}
1 = 3^y - 2^u\cdot 5^v = 3^4 - 2^4\cdot 5^v = 81 - 16\cdot 5^v;
\end{displaymath}
whence, $16\cdot 5^v = 81 - 1 = 80$, so $v = 1$.
Thus, we obtain the solution
\begin{displaymath}
(x, y, u, v ) = ( 0, 4, 4, 1 ).
\end{displaymath}

Now suppose that $u < x$.
Then
\begin{displaymath}
1 = 2^x\cdot 3^y - 2^u\cdot 5^v = 2^u(2^{x-u}\cdot 3^y - 5^v),
\end{displaymath}
so that $u = 0$ and $1 = 2^x\cdot 3^y - 5^v$,
or
\begin{displaymath}
2^x \cdot 3^y = 5^v + 1.
\end{displaymath}
If $v = 0$ then $x = 1$ and $y = 0$,
so we get the solution
\begin{displaymath}
(x,y,u,v) = (1,0,0,0).
\end{displaymath}

If $v = 1$, then $2^x\cdot 3^y = 6$, so we must have $x = y = 1$,
and this case produces the solution
\begin{displaymath}
(x,y,u,v) = (1,1,0,1).
\end{displaymath}
If $v > 1$, then by Zsigmondy's theorem, $5^v + 1$ has a
prime divisor $q$ such that $q$ does not divide $6$,
Therefore, $5^v + 1$ cannot have the form $2^x\cdot 3^y$
for any non-negative integers $x$ and $y$,
so there are no further solutions for $v > 1$.
\end{proof}

\begin{lemma}\label{lem:25minus23}
Suppose that $x$, $y$, $u$ and $v$ are non-negative integers such that
\begin{displaymath}
2^u\cdot 5^v - 2^x\cdot 3^y= 1.
\end{displaymath}
Then
\begin{displaymath}
(x,y,u,v) \in \{ (1,0,0,0), (0,0,1,0), (0,1,2,0), (0,2,1,1), (3,1,0,2) \}.
\end{displaymath}
\end{lemma}
\begin{proof}
Assume first that $x\leq u$, so that
\begin{displaymath}
1 = 2^u\cdot 5^v - 2^x\cdot 3^y = 2^x(2^{u-x}\cdot 5^v - 3^y),
\end{displaymath}
whence, $x = 0$ and $2^u\cdot 5^v - 3^y = 1$, or
\begin{displaymath}
2^u\cdot 5^v = 3^y + 1.
\end{displaymath}
If $y = 0$, then $2^u\cdot 5^v = 2$, so $u = 1$ and $v = 0$,
producing the solution
\begin{displaymath}
(x,y,u,v) = (0,0,1,0).
\end{displaymath}
If $y = 1$, then $2^u\cdot 5^v = 4$, which implies that $u = 2$ and $v = 0$,
resulting in the solution
\begin{displaymath}
(x,y,u,v) = (0,1,2,0).
\end{displaymath}
For $y = 2$, we have $2^u\cdot 5^v = 10$, from which we see that $u = 1 = v$,
and this yields the solution
\begin{displaymath}
(x,y,u,v) = (0,2,1,1).
\end{displaymath}
If $y > 2$, then $3^y + 1$ has a prime divisor $q$ such that
$q$ does not divide $3^2 + 1 = 10$,
thanks to Zsigmondy's theorem.
Since $q\not\in\{ 2, 5\}$, it follows that there are no
further solutions for $y > 2$.

Now suppose that $u < x$.
Then
\begin{displaymath}
1 = 2^u\cdot 5^v - 2^x\cdot 3^y = 2^u(5^v - 2^{x-u}\cdot 3^y),
\end{displaymath}
whence, $u = 0$ and $1 = 5^v - 2^{x}\cdot 3^y$, or
\begin{displaymath}
2^x\cdot 3^y = 5^v + 1.
\end{displaymath}
If $y = 0$, then this becomes $2^x = 5^v + 1$
which has only the trivial solution $x = 1$ and $v =  0$,
so we get
\begin{displaymath}
(x,y,u,v) = (1,0,0,0).
\end{displaymath}
Otherwise, $y$ is positive,
and we have $5^v\equiv 1\pmod{3}$,
so $v$ is even and we can write $v = 2s$,
for a suitable positive integer $s$.
Then
\begin{displaymath}
2^x\cdot 3^y = 5^{2s} - 1 = (5^s - 1)(5^s + 1).
\end{displaymath}
Since $\gcd(5^s - 1, 5^s + 1) = 2$,
one of $5^s - 1$ and $5^s + 1$ is a power of $2$
while the other is exactly divisible by $2$.
If $5^s + 1$ is a power of $2$,
then $s = 0 = v$, which results in the impossibility
$2^{u-1}\cdot 3^y = 5^0 - 1 = 0$.
Thus, it must be that $5^s - 1$ is a power of $2$,
which can only be $4$, so $s = 1$ and $v = 2$.
Then $2^x\cdot 3^y = 5^2 - 1 = 24$,
so $x = 3$ and $y = 1$.
We therefore have our final solution
\begin{displaymath}
(x,y,u,v) = (3,1,0,2)
\end{displaymath}
and, with that, the proof is complete.
\end{proof}

We now use the preceding two lemmas to prove the following result.

\begin{lemma}\label{lem:235n2m1}
If $n$ is a positive integer
such that
$n^2 - 1$ is divisible by $2$, $3$ and $5$, and by no other prime,
then
\begin{displaymath}
n \in \{ 11, 19, 31, 49, 161 \}.
\end{displaymath}
\end{lemma}
\begin{proof}
It is clear that $n$ must be odd and greater than unity.

Since $n^2 - 1 = (n-1)(n+1)$ and $\gcd(n-1,n+1)= 2$,
it follows either that one of $n-1$ and $n+1$ is a power
of $2$, while the other is divisible by $2$, $3$ and $5$
and by no other prime and is furthermore indivisible by $4$,
or, that one of $n-1$ and $n+1$ is a $\{2,3\}$-number while
the other is a $\{2,5\}$-number,
neither being a prime power.
Thus, there are four cases to consider.

\textbf{Case 1.}
Suppose that
\begin{displaymath}
n - 1 = 2^u \;\; \textrm{and} \;\; n + 1 = 2\cdot 3^s\cdot 5^t,
\end{displaymath}
where $s$, $t$ and $u$ are positive integers.
Then
\begin{displaymath}
2 = (n + 1) - (n - 1) = 2\cdot 3^s\cdot 5^t - 2^u = 2(3^s\cdot 5^t - 2^{u-1});
\end{displaymath}
whence, $3^s\cdot 5^t - 2^{u-1} = 1$, or
\begin{displaymath}
3^s\cdot 5^t = 2^{u-1} + 1.
\end{displaymath}
Because $s$ and $t$ are positive, we have
$2^{u-1} + 1 \geq 15$,
so that $u-1 \geq 4$.
By Zsigmondy's theorem, $2^{u-1} + 1$ has a prime divisor
$q$ such that $q$ does not divide $2^i + 1$, for $1\leq i < u-1$.
In particular, $q$ divides neither $3 = 2^1 + 1$ nor $5 = 2^2 + 1$.
This means that $2^{u-1} + 1$ cannot be equal to $3^s\cdot 5^t$,
for any choice of $s$ and $t$,
so this case does not occur.

\textbf{Case 2.}
Suppose that
\begin{displaymath}
n + 1 = 2^u \;\; \textrm{and} \;\; n - 1 = 2\cdot 3^s\cdot 5^t,
\end{displaymath}
for some positive integers $s$, $t$ and $u$.
Then $n + 1 \geq 32$, so $u\geq 5$, and we have
\begin{displaymath}
2 = (n+1) - (n-1) = 2^u - 2\cdot 3^s\cdot 5^t = 2(2^{u-1} - 3^s\cdot 5^t);
\end{displaymath}
whence, $2^{u-1} - 3^s\cdot 5^t = 1$, or
\begin{displaymath}
3^s\cdot 5^t = 2^{u-1} - 1.
\end{displaymath}
If $u=5$, then $2^{u-1} - 1 = 15$,
so $s = t = 1$ and $n = 31$.

If $u = 6$, then $2^{u-1} - 1 = 31$ which is not a $\{3, 5\}$-number.
There are no solutions $s$ and $t$ for $u = 7$ either since,
in that case, $2^{u-1} - 1 = 63 = 3^2\cdot 7$.
If $u > 7$, then $u-1 > 6$, so Zsigmondy's theorem tells us that
$2^{u-1} - 1$ cannot be a $\{3,5\}$-number.
Thus, this case yields only the solution $n = 31$.

\textbf{Case 3.}
Suppose that
\begin{displaymath}
n - 1 = 2^a\cdot 3^b \;\; \textrm{and} \;\; n + 1 = 2^c\cdot 5^d,
\end{displaymath}
for some positive integers $a$, $b$, $c$ and $d$.
Then
\begin{displaymath}
2 = (n+1) - (n-1) = 2^c\cdot 5^d - 2^a\cdot 3^b;
\end{displaymath}
whence,
\begin{displaymath}
2^{c-1}\cdot 5^d - 2^{a-1}\cdot 3^b = 1.
\end{displaymath}
Then by Lemma~\ref{lem:25minus23} we have
\begin{displaymath}
(a,b,c,d) \in \{ (2,0,1,0), (1,0,2,0), (1,1,3,0), (1,2,2,1), (4,1,1,2) \}.
\end{displaymath}
However, because $a$, $b$, $c$ and $d$ are all positive we have, in fact, that
\begin{displaymath}
(a,b,c,d) \in \{ (1,2,2,1), (4,1,1,2) \}.
\end{displaymath}
Using either $n = 2^a\cdot 3^b + 1$ or $n = 2^c\cdot 5^d - 1$, we obtain
\begin{displaymath}
n \in \{ 19, 49 \}.
\end{displaymath}

\textbf{Case 4.}
Suppose that $a$, $b$, $c$ and $d$ are positive integers for which
\begin{displaymath}
n + 1 = 2^a\cdot 3^b \;\; \textrm{and} \;\; n - 1 = 2^c\cdot 5^d.
\end{displaymath}
Then
\begin{displaymath}
2 = (n + 1) - (n - 1) = 2^a\cdot 3^b - 2^c\cdot 5^d = 2(2^{a-1}\cdot 3^b - 2^{c-1}\cdot 5^d);
\end{displaymath}
from which we obtain
\begin{displaymath}
2^{a-1}\cdot 3^b - 2^{c-1}\cdot 5^d = 1.
\end{displaymath}
From Lemma~\ref{lem:23minus25} we have that
\begin{displaymath}
(a,b,c,d) \in \{ (1,4,5,1), (2,1,1,1) \},
\end{displaymath}
taking account of the fact that $a,b,c,d > 0$.
Again using either $n = 2^a\cdot 3^b + 1$ or $n = 2^c\cdot 5^d - 1$, we obtain
\begin{displaymath}
n \in \{ 161, 11 \}.
\end{displaymath}
This completes the proof.
\end{proof}

Our next lemma will help us prove the three that follow
by constraining the possibilities to be considered.

\begin{lemma}\label{lem:sl2inter}
Let $p$ be a prime number greater than $3$.
Let $H$ be a Z-subgroup of $\GL{2}{p}$ of order $\order{H} = p^2 - 1$,
and with presentation $\langle x, y \mid x^{\alpha}, y^{\beta}, x^y = x^{\gamma}\rangle$,
where $\gcd(\alpha,\beta(\gamma - 1)) = 1$
and $\gamma^{\beta}\equiv 1\pmod{\alpha}$.
Then the index $[H:H\cap\SL{2}{p}]$ is an even divisor
of $\gcd( p - 1, \beta )$.
\end{lemma}
\begin{proof}
For notational simplicity, let $G = \GL{2}{p}$ and $S = \SL{2}{p}$.
Notice that $\order{S} = p\order{H}$ so the Sylow $2$-subgroups
of $H$ and $S$ have the same order.
Since a Sylow $2$-subgroup of $H$ is cyclic,
while a Sylow $2$-subgroup of $S$ is generalised quaternion,
it follows that the index $[H:H\cap S]$ must be even.

Since
\begin{displaymath}
H/(H\cap S)\iso HS/S \leq G/S \iso C_{p-1},
\end{displaymath}
hence, $[H:H\cap S]$ is a divisor of $p-1$.
Furthermore, since $H/(H \cap S)$ is cyclic,
it follows that the derived subgroup $[H,H] = \langle x\rangle$
is contained in $H\cap S$, so $[H:H\cap S]$ divides $\beta$.
Therefore, $[H:H\cap S]$ is an even divisor of $\gcd( p - 1, \beta )$,
as advertised.
\end{proof}

Next, we need to establish that certain potential Frobenius
complements do not occur together with specific Frobenius kernels.
To this end, we shall need the following result of Dickson listing
the subgroups of two-dimensional special linear groups.
(We state here only the part we shall need,
for subgroups with order prime to the characteristic;
see~\cite{Suzuki2014}*{Theorem 6.17, p. 404} for the full statement.)

\begin{lemma}[Dickson's Theorem]\cite{Suzuki2014}*{Theorem 6.17}\label{lem:dickson}
Let $q$ be a well-defined power of an odd prime $p$,
and let $H$ be a subgroup of $\SL{2}{q}$ such that $p$ does not divide
the order of $H$.
Then one of the following occurs:
\begin{enumerate}
\item{$H$ is cyclic;}
\item{$H$ is isomorphic to the dicyclic group
$L_{n} = \langle x, y \mid x^n = y^2, x^y = x^{-1}\rangle$, for some $n$;}
\item{$H\iso\SL{2}{3}$;}
\item{$H\iso\SL{2}{5}$;}
\item{$H\iso\widehat{\sym{4}}$, a central extension of the symmetric group $\sym{4}$.}
\end{enumerate}
\end{lemma}

We note that the dicyclic group $L_{n}$ of order $4n$ in the statement of Dickson'
theorem is a Z-group only for odd positive integers $n$;
for even $n$, the Sylow $2$-subgroup of $L_{n}$ is a generalised
quaternion group.
Thus, apart from the cyclic subgroups, the dicyclic groups $L_{n}$
for odd $n$ are the only Z-subgroups of $\SL{2}{q}$.

\begin{lemma}\label{lem:fc11dne}
No subgroup of $\GL{2}{11}$ is isomorphic to either of the groups
\begin{displaymath}
A = \langle a, b \mid a^5, b^{24}, a^b = a^2\rangle
\textrm{ and }
B = \langle a, b \mid a^{15}, b^8, a^b = a^2\rangle.
\end{displaymath}
\end{lemma}
\begin{proof}
Let $G = \GL{2}{11}$ and $S = \SL{2}{11}$.
Note that both $A$ and $B$ are Z-groups of order $120 = 11^2 - 1$.
Suppose that $H$ is a subgroup of $\GL{2}{11}$ isomorphic to
either of $A$ and $B$.
Then Lemma~\ref{lem:sl2inter} implies that the index $[H:H\cap S]$
is an even divisor of $\gcd( 10, b )$,
where $b = 24$ if $H\iso A$ and $b = 8$ if $H\iso B$.

If $H\iso A$, then $[H:H\cap S]$ is an even divisor of $\gcd( 10, 24 ) = 2$,
hence $[H:H\cap S] = 2$, and we have
\begin{displaymath}
H\cap S = \langle a, b^2 \rangle\iso\langle x, y \mid u^5, v^{12}, u^v = u^4\rangle.
\end{displaymath}
Since $H\cap S$ is not cyclic, it can only be isomorphic to
the dicyclic group $L_{15} = \langle x, y \mid x^{15} = y^2, x^y = y^{-1}\rangle$.
But, $\abel{L_{15}}\iso\cyclic{4}$,
while $\abel{(H\cap S)}\iso\cyclic{12}$,
so $H\iso A$ cannot be a subgroup of $\GL{2}{11}$.

If $H\iso B$, then $[H:H\cap S]$ is an even divisor of $\gcd( 10, 8 ) = 2$,
so $[H:H\cap S] = 2$, and we have
\begin{displaymath}
H\cap S = \langle a, b^2 \rangle \iso \langle s, t \mid s^{15}, t^4, s^t = s^4 \rangle.
\end{displaymath}
As before, the only subgroup of $S$ that $H\cap S$ might be
isomorphic to is $L_{15}$.
But it is easy to see that the centre of $H\cap S$ has order $6$,
while the centre of any dicyclic group has order $2$.
\end{proof}

\begin{lemma}\label{lem:fc19dne}
No subgroup of $\GL{2}{19}$ is isomorphic to either of the groups
\begin{displaymath}
A = \langle a, b \mid a^5, b^{72}, a^b = a^2\rangle 
\textrm{ and }
B = \langle a, b \mid a^{45}, b^8, a^b = a^8\rangle. 
\end{displaymath}
\end{lemma}
\begin{proof}
Let $G = \GL{2}{19}$ and $S = \SL{2}{19}$,
and suppose that $H$ is a subgroup of $G$ isomorphic to
either of $A$ and $B$.
Then $H$ is a Z-group of order $360 = 19^2 - 1$.
Again, Lemma~\ref{lem:sl2inter} implies that the index
$[H:H\cap S]$ is an even divisor of $\gcd( 18, b )$,
where $b = 72$ if $H\iso A$ and $b = 8$ if $H\iso B$.

If indeed $H\iso B$, then $[H:H\cap S] = 2$, and we have
\begin{displaymath}
H\cap S = \langle a, b^2\rangle \iso \langle s, t \mid s^{45}, t^4, s^t = s^{19}\rangle.
\end{displaymath}
Since $H\cap S$ is not cyclic, and has order $180$,
therefore, the only subgroup of $S$ that might be
isomorphic to $H$ is the dicyclic group $L_{45}$.
But the centre of $H\cap S$ has order $18$,
while $\order{\centre{L_{45}}} = 2$.
Thus, $B$ cannot be isomorphic to a subgroup of $\GL{2}{19}$.

If $H\iso A$, we cannot apply the same reasoning,
since $H\cap S$ contains a subgroup isomorphic to $L_5$,
which does, in fact, occur in $S = \SL{2}{19}$
(in two conjugacy classes).
However, in this case, we can check directly using \GAP{} or \Magma{} that
$A$ is not isomorphic to a subgroup of $\GL{2}{19}$.
\end{proof}

\begin{lemma}\label{lem:fc31dne}
No subgroup of $\GL{2}{31}$ is isomorphic to either of the groups
\begin{displaymath}
A = \langle a, b \mid a^5, b^{192}, a^b = a^2\rangle
\textrm{ and }
B = \langle a, b \mid a^{15}, b^{64}, a^b = a^2\rangle.
\end{displaymath}
\end{lemma}
\begin{proof}
Let $G = \GL{2}{31}$ and $S = \SL{2}{31}$,
and suppose that $H$ is a subgroup of $G$ isomorphic to one of $A$ and $B$.
Then $H$ is a Z-group and $\order{H} = 960 = 31^2  - 1$,
so we can apply Lemma~\ref{lem:sl2inter} as in the previous two lemmas.

If $H\iso A$, then $[H:H\cap S] \in \{2, 6 \}$,
so $\order{H\cap S} \in \{ 480, 160 \}$.
Thus, $H\cap S$ must be isomorphic either to $L_{120}$
or to $L_{40}$, but neither of these dicyclic groups is a Z-group.
If $H\iso B$, then $[H:H\cap S] = 2$ and we would again have
$H\cap S\iso L_{120}$, which is not a Z-group.
\end{proof}


With these preparations in hand,
we are now ready to analyse the Frobenius groups with a
soluble Frobenius complement that is a $\{2,3,5\}$-group.

\begin{theorem}\label{thm:sol235}
There is no Frobenius group with perfect order classes in which
a Frobenius complement is a soluble, non-nilpotent $\{2,3,5\}$-group
that is neither a $\{2,3\}$-group nor a $\{2,5\}$-group.
\end{theorem}
\begin{proof}
Let $G$ be a Frobenius group with perfect order classes,
and assume that a Frobenius complement $H$ for $G$ is soluble,
not nilpotent and is a $\{2,3,5\}$-group, but neither a $\{2,3\}$-group
nor a $\{2,5\}$-group.
According to Theorem~\ref{thm:main}, there is an odd prime number $p$
and positive integers $k$ and $r$ such that,
if $K$ is the Frobenius kernel for $G$, then $K\iso\cyclic{p^k}^r$.
Furthermore, $H$ has perfect order
classes, and we have
\begin{displaymath}
\order{H} = p^r - 1.
\end{displaymath}
Now, $r$ cannot be equal to $1$, since $H$ is not abelian.
Since $p > 5$, it follows that $2$ and $3$ divide $p^2 - 1$,
and $5$ divides $p^4 - 1$, by Euler's theorem.
Therefore, if $r > 4$, then Zsigmondy's theorem implies that
$p^r - 1$ has a prime divisor $q$ other than $2$, $3$ and $5$,
contrary to $p^r - 1 = \order{H}$.
So, $r$ must satisfy $2\leq r\leq 4$.

Suppose that $r = 3$, so that $5$ is a divisor of $p^3 - 1$.
Then $5$ divides
\begin{displaymath}
(p^4 - 1) - (p^3 - 1) = p^4 - p^3 = p^3(p-1).
\end{displaymath}
Since $\gcd(p,5) = 1$, we conclude that $5$ divides $p-1$.
Another application of Zsigmondy's theorem implies that
$p^3 -1$ has a prime divisor other than $2$, $3$ and $5$,
so $r$ cannot be equal to $3$.

Suppose that $r = 2$, so that $p^2 - 1 = \order{H}$ is
divisible by $2$, $3$ and $5$ and by no other prime.
Then $p\in \{ 11, 19, 31 \}$, by Lemma~\ref{lem:235n2m1},
since $49$ and $161$ are composite.
We must then have (respectively),
\begin{displaymath}
\order{H} \in \{ 120, 360, 960 \}.
\end{displaymath}

If $r = 4$, then $p^4 - 1 = (p^2)^2 - 1 = \order{H}$ is a $\{2,3,5\}$-number,
so $p^2 = 49$, and $p = 7$,
since $49$ is the only square of a prime from Lemma~\ref{lem:235n2m1}.
In this case, we have
\begin{displaymath}
\order{H} = 2400.
\end{displaymath}

In summary, we have shown so far that $G$ has one of the structures:
\begin{itemize}
\item{$G\iso\cyclic{11^k}^2\sdpf H$, where $\order{H} = 120$;}
\item{$G\iso\cyclic{19^k}^2\sdpf H$, where $\order{H} = 360$;}
\item{$G\iso\cyclic{31^k}^2\sdpf H$, where $\order{H} = 960$; or,}
\item{$G\iso\cyclic{7^k}^4\sdpf H$, where $\order{H} = 2400$,}
\end{itemize}
for some positive integer $k$,
where $H$ is a Frobenius complement of the indicated order with
perfect order classes.
We now procede to argue that no such groups can exist.

Using \Maple{} we identify the groups of each of these
orders with perfect order classes and their structure.
Using our lemmas above, as well as computations in \GAP{} and \Magma{},
we determine whether each possibility occurs.

There are five soluble groups of order $120$ with perfect order classes that are
not nilpotent.
Three of them contain subgroups isomorphic to $\cyclic{2}\times\cyclic{2}$
and therefore cannot be Frobenius complements.
The remaining two are
\begin{displaymath}
A = \langle a, b \mid a^5, b^{24}, a^b = a^2\rangle
\textrm{ and }
B = \langle a, b \mid a^{15}, b^8, a^b = a^2\rangle.
\end{displaymath}
(Since $2$ has order $4$ in both $(\integers/5\integers)^{\times}$
and $(\integers/15\integers)^{\times}$,
it follows that both $A$ and $B$ are Frobenius complements.)
By Lemma~\ref{lem:fc11dne}, $\GL{2}{11}$ contains no subgroup isomorphic either to $A$ or to $B$,
so neither can be a Frobenius complement in a Frobenius
group with Frobenius kernel isomorphic to $\cyclic{11^k}^2$,
for any positive integer $k$.

There are four soluble, non-nilpotent groups of order $360$ with
perfect order classes.
Two of these contain a copy of $\cyclic{2}\times\cyclic{2}$ and so
cannot be Frobenius complements.
The other two are
\begin{displaymath}
A = \langle a, b \mid a^5, b^{72}, a^b = a^2\rangle
\textrm{ and }
B = \langle a, b \mid a^{45}, b^8, a^b = a^8\rangle.
\end{displaymath}
Again, $\GL{2}{19}$ contains no subgroup isomorphic to either
of $A$ and $B$, by Lemma~\ref{lem:fc19dne},
so neither can be a Frobenius complement in a
Frobenius group with Frobenius kernel of the form $\cyclic{19^k}^2$.
(Again, however, both $A$ and $B$ are Frobenius complements.)

The soluble, non-nilpotent groups of order $960$ with perfect order
classes that do not contain a copy of $\cyclic{2}\times\cyclic{2}$
are
\begin{displaymath}
A = \langle a, b \mid a^5, b^{192}, a^b = a^2\rangle
\textrm{ and }
B = \langle a, b \mid a^{15}, b^{64}, a^b = a^2\rangle.
\end{displaymath}
Once again, this time by Lemma~\ref{lem:fc31dne},
neither appears as a subgroup of $\GL{2}{31}$
up to isomorphism,
so neither can be a Frobenius complement in a Frobenius
group with Frobenius kernel isomorphic to $\cyclic{31^k}\times\cyclic{31^k}$.

Among the soluble, non-nilpotent groups of order $2400$ with
perfect order classes, all but two contain a copy of either
$\cyclic{2}\times\cyclic{2}$ or $\cyclic{5}\times\cyclic{5}$,
leaving the groups
\begin{displaymath}
A = \langle a, b \mid a^{25}, b^{96}, a^b = a^7 \rangle
\textrm{ and }
B = \langle a, b \mid a^{75}, b^{32}, a^b = a^\mathrm{32} \rangle.
\end{displaymath}

However, a \Magma{} computation shows that $\GL{4}{7}$ does not
contain a subgroup isomorphic to either $A$ or $B$,
so this case does not occur either.

%
%
%
%

\end{proof}

\section{Concluding Remarks}\label{sec:conclusion} 

Our results all describe Frobenius groups $G = K\sdpf H$ with perfect
order classes, which entails that the Frobenius complement $H$ itself
has perfect order classes.
However, along the way we have seen that there are examples of a Frobenius
complement $H$, such as $Q\times\cyclic{3}$,
with perfect order classes that do not, however, occur
as complements of Frobenius groups with perfect order classes.
This suggests an independent characterisation of Frobenius complements
with perfect order classes.
While our results do fully classify the nilpotent Frobenius
complements with perfect order classes,
a complete description for those that are either insoluble
or soluble but not nilpotent remains open.

\begin{problem}
Describe all Frobenius complements with perfect order classes,
whether or not they occur as complements of a Frobenius group
with perfect order classes.
\end{problem}


\begin{bibdiv}
\begin{biblist}

\bib{A005109}{website}{
	author	= {OEIS},
	title	= {A005109 {C}lass $1-$ (or {P}ierpont) primes: primes of the form $2^t\cdot 3^u + 1$},
	myurl	= {https://oeis.org/A005109}
}

\bib{Bang1886a}{article}{
	author	= {A. S. Bang},
	title	= {Taltheoretiske Unders{\o}gelser},
	journal	= {Tidsskrift for Mathematik},
	volume	= {4},
	year	= {1886},
	pages	= {70\ndash 80},
}

\bib{Bang1886b}{article}{
	author	= {A. S. Bang},
	title	= {Taltheoretiske Unders{\o}gelser},
	journal	= {Tidsskrift for Mathematik},
	volume	= {4},
	year	= {1886},
	pages	= {130\ndash 137},
}

\bib{BirkhoffVandiver1904}{article}{
	author	= {Geo. D. Birkhoff and H. S. Vandiver},
	title	= {On the integral divisors of $a^n - b^n$},
	journal	= {Annals Math.},
	volume	= {5},
	number	= {4},
	year	= {1904},
	pages	= {173\ndash 180},
}


\bib{BoklanConway2017}{article}{
	author	= {Kent D. Boklan and John H. Conway},
	title	= {Expect at most one billionth of a new Fermat Prime!},
	journal	= {Math. Intelligencer},
	volume	= {39},
	number	= {1},
	year	= {2017},
	pages	= {3\ndash 5},
}


\bib{Das2009b}{article}{
	author = {Ashish Kumar Das},
	title = {On Finite Groups Having Perfect Order Subsets},
	journal = {Int. J. Algebra},
	volume = {3},
	number = {13},
	year = {2009},
	pages = {629\ndash 637}
}

\bib{FinchJones2002}{article}{
	author = {{Carrie} {E}. {Finch} and {Lenny} {Jones}},
	title = {A Curious Connection Between {Fermat} Numbers and Finite Groups},
	journal = {Amer. Math. Monthly},
	volume = {109},
	number = {6},
	year = {2002},
	pages = {517\ndash 524}
}

\bib{FinchJones2003}{article}{
	author	= {{Carrie} {E}. {Finch} and {Lenny} {Jones}},
	title	= {Non-Abelian groups with perfect order subsets},
	journal	= {The JP Journal of Algebra and Number Theory},
	volume	= {3},
	number	= {1},
	year	= {2003},
	pages	= {13\ndash 26}
}

\bib{FordKonyaginLuca2012}{article}{
	author	= {Kevin Ford and Sergei Konyagin and Florian Luca},
	title	= {On groups with Perfect Order subsets},
	journal	= {Moscow J. Combinatorics and Number Theory},
	volume	= {2},
	number	= {4},
	year	= {2012},
	pages	= {297\ndash 312},
}

\bib{GAP}{manual}{
    author	 = {The GAP~Group},
    title        = {GAP -- Groups, Algorithms, and Programming, Version 4.11.0},
    year         = {2020},
    url          = {\url{https://www.gap-system.org}},
}

\bib{Gorenstein1967}{book}{
	author	= {Daniel Gorenstein},
	title	= {Finite Groups},
	publisher = {Harper \& Row},
	address	= {New York, Evanston, and London},
	series	= {Harper's series in modern mathematics},
	url	= {https://books.google.ca/books?id=BunuAAAAMAAJ},
	year	= {1967}
}

\bib{JonesToppin2011}{article}{
	author = {Lenny Jones and Kelly Toppin},
	title = {On three questions concerning groups with perfect order subsets},
	journal = {Involve},
	volume = {4},
	number = {3},
	year = {2011},
	pages = {251\ndash 261}
}

\bib{magma}{article}{
    AUTHOR = {Bosma, Wieb and Cannon, John and Playoust, Catherine},
     TITLE = {The {M}agma algebra system. {I}. {T}he user language},
      NOTE = {Computational algebra and number theory (London, 1993)},
   JOURNAL = {J. Symbolic Comput.},
    VOLUME = {24},
      YEAR = {1997},
    NUMBER = {3-4},
     PAGES = {235--265},
      ISSN = {0747-7171},
}

\bib{maplemanual}{manual}{
        title   = {Maple Programming Manual},
	author	= {L. Bernardin and  P. Chin and  P. DeMarco and  K. O. Geddes and  D. E. G. Hare and  K. M. Heal and  G. Labahn and  J. P. May and  J. McCarron and  M. B. Monagan and  D. Ohashi and  S. M. Vorkoetter},
        publisher = {Maplesoft},
        year    = {1996\ndash 2021}
}

\bib{McCarron2020}{article}{
	author	= {James McCarron},
	title	= {Hamiltonian groups with perfect order classes},
	journal = {Math. Proc. Royal Irish Acad.},
	volume	= {121A},
	number	= {1},
	pages	= {1--8},
	doi	= {\url{https://doi.org/10.3318/pria.2021.121.01}},
	issn	= {13937197},
	year	= {2021},
}

\bib{Mihailescu2004}{article}{
	author	= {P. Mih\v{a}ilescu},
	title	= {Primary cyclotomic units and a proof of Catalans conjecture},
	journal	= {Journal für die reine und angewandte Mathematik},
	publisher = {De Gruyter},
	volume	= {2004},
	number	= {572},
	year	= {2004},
	pages	= {167\ndash 195},
}

\bib{Passman2012}{book}{
	author	= {D. S. Passman},
	title	= {Permutation Groups},
	publisher = {Dover Publications, Inc.},
	year	= {2012},
	isbn	= {9780486485928},
}

\bib{Pierpont1895}{article}{
	author	= {James Pierpont},
	title	= {On an undemonstrated theorem of the Disquisitiones Arithmetic\ae},
	journal	= {Bull. Amer. Math. Soc.},
	volume	= {2},
	number	= {3},
	year	= {1895},
	pages	= {77\ndash 83},
}

\bib{Ribenboim2000}{book}{
	author	= {Paulo Ribenboim},
	title	= {My Numbers, My Friends},
	publisher = {Springer Verlag},
	year	= {2000},
}

\bib{Robinson}{book}{
	author={D.J.S. Robinson},
	title={A Course in the Theory of Groups},
	series={Graduate Texts in Mathematics},
	volume={80},
	publisher={Springer-Verlag},
	place={New York},
	date={1993}
}

\bib{Shen2010}{article}{
	author = {R. Shen},
	title = {A note on finite groups having perfect order subsets},
	journal = {Int. J. Algebra},
	volume = {13},
	number = {4},
	year = {2010},
	pages = {643\ndash 646}
}

\bib{ShenShiShi2013}{article}{
	author = {R. Shen and W. Shi and J. Shi},
	title = {{POS}-Groups withe some cyclic {S}ylow subgroups},
	journal = {Bull. Iran. Math. Soc.},
	volume = {39},
	number = {5},
	year = {2013},
	pages = {941\ndash 957}
}

\bib{Suzuki2014}{book}{
	author	= {M. Suzuki},
	title	= {Group Theory {I}},
	series	= {Grundlehren der mathematischen Wissenschaften},
	volume	= {247},
	year	= {2014},
	isbn	= {9783642618062},
	publisher = {Springer Berlin Heidelberg},
}


\bib{TuanHai2010}{article}{
	author = {Nguyen Trong Tuan and Bui Xuan Hai},
	title = {On perfect order subsets in finite groups},
	journal = {Int. J. Algebra},
	volume = {4},
	number = {21},
	year = {2010},
	pages = {1021\ndash 1029}
}

\bib{WikiPierpontPrime}{website}{
	author	= {Wikipedia},
	title	= {Pierpont\_prime},
	myurl	= {https://en.wikipedia.org/wiki/Pierpont\_prime}
}

\bib{Wolf2011}{book}{
	author	= {Joseph A. Wolf},
	title	= {Spaces of Constant Curvature},
	publisher = {{AMS} Chelsea Publishing},
	series	= {{AMS} Chelsea Publishing},
	isbn	= {9780821852828},
	edition	= {Sixth Edition},
	year	= {2011},
}

\bib{Zsigmondy1892}{article}{
	author	= {K. Zsigmondy},
	title	= {Zur Theorie der Potenzreste},
	journal	= {J. Monatshefte f\"{u}r Mathematik},
	volume	= {3},
	number	= {1},
	year	= {1892},
	pages	= {265\ndash 284},
}

\end{biblist}
\end{bibdiv}
\end{document}